\documentclass[11pt, reqno]{amsart}
\usepackage[margin=3.6cm]{geometry}
\usepackage{array,booktabs,tabularx}
\usepackage{graphicx}
\usepackage{amssymb}
\usepackage{enumerate}
\usepackage{color}
\usepackage{hyperref}
\usepackage{esint}
\usepackage{mathtools}
\usepackage{dsfont}
\usepackage{ esint }
\usepackage[show]{ed}
\usepackage{comment}
\usepackage{enumitem}
\usepackage{soul}
\usepackage{needspace}

\newtheorem{theorem}{Theorem}
\newtheorem{corollary}[theorem]{Corollary}

\newtheorem{lemma}[theorem]{Lemma}
\newtheorem{proposition}[theorem]{Proposition}

\newtheorem*{conjecture}{Conjecture}

\theoremstyle{definition}

\newtheorem{remark}{Remark}

\newcommand{\R}{{\mathbb R}}

\newcommand{\ep}{\varepsilon}

\newcommand{\cs}{$\clubsuit$}

\newcommand{\mc}[1]{\mathcal{#1}}
\newcommand{\re}{\mathbb{R}}
\newcommand{\Span}{\operatorname{span}}

\newcommand{\red}[1]{{\color{red}{#1}}}

\newcommand{\SNH}{S\!N^*\!H}
\newcommand{\sn}{\SNH}
\newcommand{\sig}{\sigma_{_{\!\!S\!N^*\!H}}}

\newcommand{\snx}{S\!N^*_xH}

\newcommand{\SM}{S^*\!M}
\newcommand{\Sig}{\sigma_{_{\!\Sigma_{H\!,p}}}}
\newcommand{\SigH}{\Sigma_{_{\!H\!,p}}}
\newcommand{\comp}{\operatorname{comp}}
\newcommand{\muH}{\mu_{_{\! H}}}
\newcommand{\LambdaH}{\Lambda_{_{\!H\!,T}}}
\newcommand{\LambdaHtwo}{\Lambda_{_{\!H\!,2T}}}
\newcommand{\we}{\iota_{w,\e}}
\newcommand{\twe}{\tilde{\iota}_{w,\e}}

\def\XXint#1#2#3{{\setbox0=\hbox{$#1{#2#3}{\int}$} \vcenter{\hbox{$#2#3$}}\kern-.5\wd0}}

\DeclareMathOperator{\vol}{vol}

\DeclareMathOperator{\supp}{supp}

\newcommand{\dbox}{\dim_{\operatorname{box}}}
\newcommand{\e}{\varepsilon}

\title[Eigenfunction averages]{
On the growth of eigenfunction averages: microlocalization and geometry
}

\author{Yaiza Canzani}
\address{Department of Mathematics, University of North Carolina, Chapel Hill, NC, USA}
\email{canzani@email.unc.edu}
\author{Jeffrey Galkowski}
\address{Department of Mathematics, Stanford University, Stanford, CA, USA}
\email{jeffrey.galkowski@stanford.edu }

\begin{document}

\begin{abstract}
Let $(M,g)$ be a smooth, compact Riemannian manifold and $\{\phi_h\}$ an $L^2$-normalized sequence of Laplace eigenfunctions, {$-h^2\Delta_g\phi_h=\phi_h$. Given  a smooth submanifold $H \subset M$  of codimension $k\geq 1$, we find conditions on the pair $(\{\phi_h\},H)$ for which
$$
\Big|\int_H\phi_hd\sigma_H\Big|=o(h^{\frac{1-k}{2}}),\qquad h\to 0^+.
$$
One such condition is that the set of conormal directions to $H$ that are recurrent {has} measure $0$. In particular, we show that the upper bound holds for any $H$ if $(M,g)$ is surface with Anosov geodesic flow or a manifold of constant negative curvature. The results are obtained by {characterizing} the behavior of the defect measures {of eigenfunctions with maximal averages.}
}
\end{abstract}

\maketitle
\section{Introduction}
On a compact Riemannian manifold $(M,g)$ of dimension $n$ we consider sequences of Laplace eigenfunctions $\{\phi_h\}$ solving 
\[
(-h^2\Delta_g-1)\phi_h=0,\qquad{\|\phi_h\|_{L^2(M)}=1.}
\]
In this article, we study the average oscillatory behavior of $\phi_h$ when restricted to a submanifold $H\subset M$. In particular, we seek to understand conditions on the pair $(\{\phi_h\},H)$ under which
\begin{equation}
\label{e:average}
\int_H\phi_hd\sigma_H =o\big(h^{\frac{1-k}{2}}\big), 
\end{equation}
as $h \to 0^+$, where $\sigma_H$ is the volume measure on $H$ induced by the Riemannian metric, and $k$ is the codimension of $H$. 

We note that the bound 
\begin{equation}
\label{e:StdBound}
\Big| \int_H\phi_hd\sigma_H\Big| = O(h^{\frac{1-k}{2}})
\end{equation}
holds for any pair $(\{\phi_h\},H)$~\cite[Corollary 3.3]{Zel}, and is sharp in general. Therefore, we seek to give conditions under which the average is sub-maximal. Integrals of the form~\eqref{e:average}, where $H$ is a curve,  have a long history of study. Good~\cite{Good} and Hejhal~\cite{Hej} study the case in which $H$ is a periodic geodesic in a compact hyperbolic manifold, and prove the bound~\eqref{e:StdBound} in that case. The work of Zelditch~\cite{Zel} in fact shows that~\eqref{e:average} holds for a density one subsequence of eigenvalues. Moreover, one can give explicit polynomial improvements on the error term in~\eqref{e:StdBound} for a density one subsequence of eigenfunctions~\cite{JZ}.

These estimates, however, are not {generally} satisfied for the full sequence of eigenfunctions and the question of when all eigenfunctions satisfy~\eqref{e:average} has been studied recently for the case of curves in surfaces~\cite{CS,SXZ, Wym, Wym2} and for submanifolds~\cite{Wym3}. Finally, given a hypersurface, the question of which eigenfunctions satisfy~\eqref{e:average} was studied in~\cite{CGT}. In this article, we address both of these questions, strengthening the results concerning which eigenfunctions can have maximal averages on a given submanifold $H$, and giving weaker conditions on the submanifold $H$ that guarantee that~\eqref{e:average} holds for all eigenfunctions.

This article improves and extends nearly all existing results regarding averages of eigenfunctions over submanifolds. We recover all conditions guaranteeing that the improved bound \eqref{e:average} holds found in~\cite{CS,SXZ,Wym,Wym2,Wym3, GT17,Gdefect,CGT,Berard77,SZ16I,SZ16II}. As far as the authors are aware, these papers contain all previously known conditions ensuring improved averages. Moreover, we give strictly weaker conditions guaranteeing \eqref{e:average} when $k<n$; we replace the condition that the set of loop directions has measure zero from~\cite{Wym3} with the condition that the set of recurrent directions has measure zero. This allows us to prove that under conditions on $(M,g)$ including those studied in~\cite{Good, Hej, CS,SXZ}, the improved bound~\eqref{e:average} holds unconditionally with respect to the submanifold $H$. These improvements are possible because the main estimate, Theorem~\ref{t:local2}, gives explicit bounds on averages over submanifolds $H$ which depend only on the microlocalization of a sequence of eigenfunctions in the conormal directions to $H$.  This gives a new proof of~\eqref{e:StdBound} from~\cite{Zel} with explicit control over the constant $C$ for high energies. In fact, we characterize those defect measures which may support maximal averages. The estimate requires no assumptions on the geometry of $H$ or $M$ and is purely local. It is only with this bound in place that we use dynamical arguments to draw conclusions about the pairs $((M,g),H)$ supporting eigenfunctions with maximal averages. We note, however, that this paper does not obtain logarithmically improved averages as in~\cite{Berard77,SXZ, Wym2}.  

Recall that all compact, negatively curved Riemannian surfaces have Anosov geodesic flow~\cite{Anosov}.  {One} consequence of the results in this paper is the following. 
 \begin{theorem}
\label{thm:2d}
Suppose $(M,g)$ is a compact, Riemannian surface with Anosov geodesic flow and {$\gamma:[a,b]\to M$} is a smooth curve segment {with $|\gamma'|>0$}. Then
\[
{\int_a^b \phi_h(\gamma(s))  ds=o(1) \qquad \text{and}\qquad \int_a^b h\partial_\nu \phi_h(\gamma(s))  ds=o(1)}
\]
as $h \to 0^+$ for every sequence $\{\phi_h\}$ of Laplace eigenfunctions. Here $\partial_\nu$ denotes the derivative in the normal direction to the curve.
\end{theorem}

In order to state our {more general} results we introduce some geometric notation. Let $H\subset M$ be a closed smooth submanifold of codimension $k$. We denote by $N^*\!H$ the conormal bundle to $H$ and we write $\sn$ for the unit conormal bundle of $H$, where the metric is induced from that in $N^*\!H \subset T^*M$. We write $\sig$ for the measure on $\sn$ induced by the Sasaki metric on $T^*M$ {(see e.g. \cite{Eberlein73})}. In particular, if $(x',x'')$ are Fermi coordinates in a tubular neighborhood of $H$, where $H$ is identified with $\{(x',x''):\,x''=0\}$, we have
\[\sig(x',\xi'')= \sigma_H(x') d{\textup{Vol}_{S^{k-1}}}(\xi''), \]
where $x=(x',0)\in H$, $\xi'' \in \snx$, and $S^{k-1}$ is the $k-1$ dimensional sphere.

 Let $T_H:\sn\to \re\cup \{\infty\}$ with  
 $$
 T_H(\rho):=\inf\{t>0:\; G^t(\rho)\in \sn\},
 $$  
be the first return time. Define the loop set
\[
\mc{L}_H:=\{\rho\in \sn:\; T_H(\rho)<\infty\}
\]
and first return map $\eta:\mc{L}_H\to \sn$ by $\eta(\rho)=G^{T_H(\rho)}(\rho).$
Next, consider the infinite loop sets 
\[
\mc{L}_H^{+\infty}:=\bigcap_{k\geq 0}\eta^{- k}(\mc{L}_H) \qquad \text{and}\qquad \mc{L}_H^{-\infty}:=\bigcap_{k\geq 0}\eta^{ k}(\mc{L}_H),
\]
and the recurrent set
\[
\mc{R}_H=\mc{R}_H^+\cap \mc{R}_H^-
\]
where
\[
\mc{R}_H^{\pm}:=\left\{\rho\in \mc{L}_H^{\pm\infty}:\;  \rho\in \bigcap_{N>0}\overline{\bigcup_{k\geq N}\eta^{\pm k}(\rho)}\right\}.
\]

In what follows we write $\pi_H:\sn \to H$ for the canonical projection map onto $H$, and $\dbox(B)$ for the Minkowski box dimension of a set $B$. 

\begin{theorem}\label{thm:recur}
Let $(M,g)$ be a smooth, compact Riemannian manifold of dimension $n$. Let $H\subset M$ be a closed embedded submanifold of codimension $k$, and $A\subset H$ be a subset with boundary $\partial A$ satisfying $\dbox(\partial A)<n-k-\frac{1}{2}. $
Suppose 
\[
\sig(\mc{R}_H\cap \pi_H^{-1}(A))=0.
\] \medskip
 Then
\[
\int_A\phi_hd\sigma_H=o(h^{\frac{1-k}{2}})
\]
as $h \to 0^+$ for every sequence $\{\phi_h\}$ of Laplace eigenfunctions.
 \end{theorem}

Theorem~\ref{thm:recur} improves on the work of Wyman~\cite{Wym3}, replacing the measure of the loop set $\mc{L}_H$, by that of the recurrent set $\mc{R}_H$. Taking $H$ to be a single point (i.e. $k=n$) also recovers the results of \cite{SoggeTothZelditch};  see Remark \ref{r:uniformity}.  

When $H$ is a hypersurface, i.e. $k=1$, we can also study the oscillatory behavior of  the normal derivative $h\partial_{\nu} \phi_h$ along $H$.
\begin{theorem}\label{thm:normal}
{Suppose $(M,g,H,A)$ satisfy the assumptions of Theorem~\ref{thm:recur} with $k=1$. Then for every sequence $\{\phi_h\}$ of Laplace eigenfunctions
\[
\left|\int_A\phi_hd\sigma_H\right| + \left|\int_A h\partial_{\nu} \phi_hd\sigma_H\right|=o(1)
\]
as $h \to 0^+$.} 
 \end{theorem}
 
Theorem~\ref{thm:recur} allows us to derive substantial conclusions about the geometry of submanifolds supporting eigenfunctions with maximal averages. Indeed, if there exists $c>0$ and  a sequence of eigenfunctions $\{\phi_h\}$ for which
\[
\left|\int_A\phi_hd\sigma_H\right|>c\,h^{\frac{1-k}{2}},
\]
then,
\[
\sig(\mc{R}_H\cap \pi_H^{-1}(A))>0.
\]

{Next, we  present different {geometric conditions on $(M,g)$ which imply} $\sig(\mc{R}_H)=0$. We recall that strictly negative sectional curvature implies Anosov geodesic flow.  Also, both Anosov geodesic flow and non-negative sectional curvature imply that $(M,g)$ has no conjugate points.}
\begin{theorem}\label{T:applications}
 Let $(M,g)$ be a smooth, compact Riemannian manifold of dimension $n$. Let $H\subset M$ be a closed embedded submanifold of codimension $k$. Suppose one of the following assumptions holds{:}
\begin{enumerate}[label=\textbf{\Alph*.},ref=\Alph*]
\item  \label{a1} $(M,g)$ has no conjugate points and $H$ has codimension $k>\frac{n +1}{2}$. \smallskip 
\item  \label{a2}$(M,g)$ has no conjugate points and $H$ is a geodesic sphere.\smallskip 
\item  \label{a3} $(M,g)$ has  constant negative curvature.\smallskip 
\item  \label{a4}$(M,g)$ is a surface with Anosov geodesic flow. \smallskip 
\item  \label{a5}$(M,g)$ has Anosov geodesic flow and {non-positive curvature}, and $H$ is totally geodesic.  \smallskip 
\item  \label{a6} $(M,g)$ has {Anosov geodesic flow} and $H$ is a subset $M$ that lifts to a  horosphere. 
\end{enumerate}
Then
\[\sig(\mc{R}_H)=0.\]
In addition, condition~\ref{a1} implies that $\sig(\mc{L}_H)=0.$
\end{theorem}

Combining Theorems \ref{thm:recur} and \ref{T:applications} gives the following result on the oscillatory behavior of eigenfunctions when restricted to $H$.

\begin{corollary}\label{C:applications}
Let $(M,g)$ be a manifold {of dimension $n$} and let $H\subset M$ be a closed embedded submanifold {of codimension $k$} satisfying one of the assumptions {\ref{a1}}-{\ref{a6}} in Theorem \ref{T:applications}. Suppose that $A\subset H$ satisfies $\dbox(\partial A)<n-k-\frac{1}{2}$. Then
\[
\int_A\phi_hd\sigma_H=o(h^{\frac{1-k}{2}})
\]
as $h \to 0^+$ for every sequence $\{\phi_h\}$ of Laplace eigenfunctions.
\end{corollary}

{We conjecture that the conclusions of Theorem~\ref{T:applications}, and hence also Corollary~\ref{C:applications}, hold in the case that $(M,g)$ is a manifold with Anosov geodesic flow of any dimension.}

{
\begin{conjecture}
Let $(M,g)$ be a manifold of dimension $n$ with Anosov geodesic flow and let $H\subset M$ be a submanifold of codimension $k$. Then
$$
\sig(\mc{R}_H)=0.
$$
\end{conjecture}
}

\subsection{Semiclassical operators and a quantitative estimate}

This section contains the key analytic theorem for controlling submanifold averages (Theorem~\ref{t:local2}) which, in particular, has Theorem{s~\ref{thm:recur} and~\ref{thm:normal} as {corollaries}}. We control the oscillatory behavior of quasimodes of semiclassical pseudodifferential operators {using} a quantitative estimate relating averages of quasimodes to the behavior of the {associated} defect measure. As a consequence, we characterize defect measures {for which} the corresponding quasimodes {may} have maximal averages.
 
  We say that a sequence of functions $\{\phi_h\}$ is \emph{compactly microlocalized} if there exists $\chi\in C_c^\infty(T^*M)$ so that 
\[
(1-Op_h(\chi))\phi_h=O_{C^\infty}(h^\infty{\|\phi_h\|_{L^2(M)}}).
\]
Also, we say that $\{\phi_h\}$ is a \emph{quasimode for $P\in \Psi^\infty_h(M)$} if
\[
P\phi_h =o_{L^2}(h),\qquad \|\phi_h\|_{L^2}=1.
\]
In addition, for $p\in S^\infty(T^*M;\R)$, we say that a submanifold $H\subset M$ of codimension $k$ is \emph{conormally transverse for $p$} if given $f_1,\dots f_{k}\in C_c^\infty(M;\R)$ {such that}
\[
H= \bigcap_{i=1}^k\{f_i=0\},\qquad\quad \{df_i\}\text{ linearly independent on }H,
\] 
we have
\begin{equation}
\label{e:transverse} N^*\!H\; \subset \; \{p\neq0\} \cup \bigcup_{i=1}^k\{ H_pf_i \neq 0\},
\end{equation}
where $H_p$ is the Hamiltonian vector field associated to $p$. 

Let 
\[
{\Sigma_p}:=\{p=0\},\qquad  \qquad {\SigH}=\Sigma_p \cap N^*\!H,
\] 
and {consider the Hamiltonian flow}
\[
\varphi_t:=\exp(tH_p).
\]
{We fix $t_0>0$ and define} for a {Borel} measure $\mu$ on $\Sigma_p$, the measure $\mu_{_{\!H,p}}$ on $\SigH$ by setting
\[
\mu_{_{\!H,p}}(A):=\frac{1}{2t_0}\mu\Big( \bigcup_{|t|\leq t_0}\varphi_t(A)\Big),\qquad A\subset \SigH\text{ {Borel}}. 
\]

Remark 2  in \cite{CGT} shows that if $\mu$ is a defect measure associated to {a quasimode} $\{\phi_h\}$ {and $H$ is conormally transverse for $p$}, then $\mu_{_{\!H,p}}(A)$ is independent of the choice of $t_0$. It is then natural to replace the fixed  choice of $t_0$ with $\lim_{t_0\to 0}$. In particular, for $\mu$ a defect measure associated to $\{\phi_h\}$,
\begin{equation} \label{e:muH2}
\mu_{_{\!H,p}}(A)= \lim_{t_0\to 0}\frac{1}{2t_0}\, \mu\Big( \bigcup_{|t|\leq t_0}\varphi_t(A)\Big),
\end{equation}
for all $A \subset \SigH$ Borel.

Next, let $r_H:M\to \re$ be the geodesic distance to $H$. That is, $r_H(x)=d(x,H)$. Then, define $|H_pr_H|:\SigH\to \re$ by 
$$
|H_pr_H|(\rho):=\lim_{t\to 0}|H_pr_H(\varphi_t(\rho))|.
$$
Finally, {we write $\mu\perp \lambda$ when $\mu$ and $\lambda$ are mutually singular measures and} let $\Sig$ be the volume measure induced on $\SigH$ by the Sasaki metric.

\begin{theorem}
\label{t:local2}
Let $(M,g)$ be a smooth, {compact} Riemannian manifold of dimension $n$ and $P\in \Psi^\infty(M)$ have real valued {principal} symbol $p(x,\xi)$. Suppose that $H\subset M$ is a {closed embedded} submanifold of codimension $k$ conormally transverse for $p$, and that $\{\phi_h\}$ is a compactly microlocalized quasimode for $P$ with defect measure $\mu$. Let  $f\in L^1(H,\Sig)$ and $\lambda_{_{\! H}}\perp \Sig$ be so that
\[
\mu_{_{\!H,p}}=f d\Sig+\lambda_{_{\! H}}.
\]
Let $w\in C^\infty(H)$ and $A \subset H$ with $\dbox(\partial A)<n-k-\frac{1}{2}$.
Then there exists $C(n,k)=C_{n,k}>0$, depending only on {$n$ and $k$}, so that
\begin{equation}
\label{e:conclusion}
\limsup_{h\to 0^+}h^{\frac{k-1}{2}}\left|\int_A w \phi_hd\sigma_H\right|\leq C_{n,k}\int_{\pi_H^{-1}(A)} |w| \sqrt{f{|H_pr_H|^{-1}}} d\Sig.
\end{equation}
\end{theorem}

\noindent In addition to relating the $L^2$ microlocalization of quasimodes to averages on submanifolds, Theorem~\ref{t:local2} gives a quantitative version of the bound \eqref{e:StdBound} proved in \cite[Corollary 3.3]{Zel} and generalizes the work of the second author \cite[Theorem 2]{Gdefect} to manifolds of any codimension. {{Note also that the estimate~\eqref{e:conclusion} is saturated for every $0<k\leq n$ on the round sphere $S^{n}$.}

{\begin{remark}
\label{r:uniformity}
It is not hard to see that we can replace~\eqref{e:conclusion} with
$$
\limsup_{h\to 0^+}h^{\frac{k-1}{2}}\sup_{(\tilde{A},\tilde{H})\in \mathcal{A}(A,H,h)}\left|\int_{\tilde{A}} \phi_hd\sigma_{\tilde{H}}\right|\leq C_{n,k}\int_{\pi_H^{-1}(A)} \sqrt{f{|H_pr_H|^{-1}}} d\Sig 
$$
where
\begin{multline*}
\mc{A}(A,H,h)=\\
\{(\tilde{A},\tilde{H})\mid \tilde{A}\subset \tilde{H},\,\dbox(\partial\tilde{A})<n-k-\frac{1}{2},\, d(A,\tilde{A})=o(1),\, d_s(\SigH,\Sigma_{\tilde{H}})=o(1)\}
\end{multline*}
and $d_s$ is the distance induced by the Sasaki metric. That is, our estimate is locally uniform in $o_{C^1}(1)$ neighborhoods of $H$  (see Remark~\ref{r:explained} for an explanation). This also implies that all of our other estimates are uniform in $o_{C^1}(1)$ neighborhoods.  
\end{remark}}

A direct consequence of Theorem~\ref{t:local2} is the following.

\begin{theorem}\label{t:local}
Let $(M,g)$ be a smooth, compact Riemannian manifold of dimension $n$. Let $H\subset M$ be a closed embedded submanifold of codimension $k$, and let $A\subset H$ be a subset with boundary $\partial A$ satisfying $\dbox(\partial A)<n-k-\frac{1}{2}.$
If $\{\phi_h\}$ is a sequence of eigenfunctions with defect measure $\mu$ so that $\muH\,\perp \,1_A\, \sig$, then 
\[
\int_A \phi_hd\sigma_H=o(h^{\frac{1-k}{2}}).
\]
\end{theorem}
 
 Theorem~\ref{t:local} strengthens the results of \cite{CGT}. In particular, in \cite{CGT}, the measure $\mu$ is said to be conormally diffuse if $\muH(\sn)=0$, which of course  implies $\muH\perp \sig.$

We note that Theorem~\ref{t:local} is an immediate consequence of Theorem~\ref{t:local2}. To see this, first observe that if we take $P=-h^2\Delta_g-1$, set $p(x, \xi)=|\xi|_{g(x)}^2-1=\sigma(P)$, and  let $\{\phi_h\}$ satisfy 
$
P\phi_h=0,
$
then
\[
(1-Op_h(\chi))\phi_h=O_{C^\infty}(h^\infty),
\]
for any  $\chi\in C_c^\infty(T^*M)$ with $\chi\equiv 1$ on $|\xi|_g\leq 2$.
Next, note that in this setting we have $\Sig=\sig$. 
Hence, if 
\[
\int_{\pi_H^{-1}(A)} \sqrt{f} d\Sig=0,
\]
then by Theorem~\ref{t:local2}, 
\[
\int_A\phi_hd\sigma_H=o(h^{\frac{1-k}{2}}).
\]
{To see that any $H\subset M$ is conormally transverse, observe} that if $H=\cap_{i=1}^k f_i$, then $N^*H=\text{span}\{df_i:\; i=1, \dots, k\}$. In particular, given $(x,\xi)\in N^*H$ there exists $i\in \{1, \dots, k\}$ for which $H_pf_i(x,\xi)=2\langle df_i(x), \xi\rangle \neq 0.$

{
\subsection{Relation with $L^\infty$ bounds}

Observe that taking $k=n$ in~\eqref{e:StdBound}, and $H=\{x\}$ for some $x\in M$ the estimate reads,
\begin{equation}
\label{e:L infinity}
|u_h(x)|\leq Ch^{\frac{1-n}{2}}.
\end{equation}
By Remark~\ref{r:uniformity} the constant $C$ can be chosen independent of $x$ (and indeed, for small $h$, depending only on the injectivity radius of $(M,g)$ and dimension of $M$ \cite{Gdefect}). 
Estimates of this form are well known, first appearing in \cite{Ava,Lev,Ho68} (see also \cite[Chapter 7]{EZB}), and situations which produce sharp examples for~\eqref{e:L infinity} are extensively studied. Many works~\cite{Berard77,I-s,TZ02,SoggeZelditch,SoggeTothZelditch,SZ16I,SZ16II} have studied connections between growth of $L^\infty$ norms of eigenfunctions and the global geometry of the manifold $M$. More recently~\cite{GT17,Gdefect} examine the relation between defect measures and $L^\infty$ norms. 

This article continues in the spirit of~\cite{GT17,Gdefect} and, in particular, taking $k=n$ in Theorem~\ref{t:local2} (together with Remark~\ref{r:uniformity}) recovers \cite[Theorem 2]{Gdefect}. Hence this article also generalizes many of the results of~\cite{SoggeZelditch, SoggeTothZelditch, SZ16I, SZ16II} to manifolds of lower codimension. For example taking $k=n$ in Theorem~\ref{thm:recur} gives the main results of~\cite{SoggeTothZelditch} (see also \cite[Corollary 1.2]{Gdefect}).
}

\subsection{Manifolds with {no focal points} or Anosov geodesic flow}
In order to prove parts \ref{a3}, \ref{a4}, \ref{a5} and \ref{a6} of Theorem \ref{T:applications}, we need to use that the underlying manifold has no focal points or Anosov {geodesic} flow. We show that these structures allow us to restrict to working on the set of points $\mc{A}_H$ in $\SNH$ at which the tangent space to $\SNH$ splits into a sum of  bounded  and unbounded directions. To make this sentence precise we introduce some notation.

If $(M,g)$ has no conjugate points, then for any $\rho \in \SM$, there exist stable and unstable subspaces $E_{\pm}(\rho)\subset T_\rho S^*M$ so that 
\[
dG^t :E_\pm(\rho) \to E_\pm(G^t(\rho))
\]
and 
\[
 |dG^t(v)|\leq C|v| \;\;  \text{for}\;  v\in E_\pm\;\; \text{and}\; t\to  \pm\infty.
\]
Moreover, {if $(M,g)$ has no focal points then} $E_\pm(\rho)$ vary continuously with $\rho$. (See for example \cite[Proposition 2.13]{Eberlein73}.)

In what follows we write  
\[
N_{\pm}(\rho):=T_{\rho}(\SNH)\cap E_\pm(\rho).
\] 
We define {the \emph{mixed} and \emph{split} subsets of $\SNH$ respectively by}
{
\begin{align*}
\mc{M}_H&:=\Big\{\rho \in \SNH:\:\, N_-(\rho)\neq \{0\}\text{ and }N_+(\rho)\neq \{0\}\Big\},\\
\mc{S}_H&:= \Big\{\rho\in \SNH:\;\, T_\rho (\SNH)=N_-(\rho)+N_+(\rho)\Big\}.
\end{align*}
}
Then we write
\begin{equation}\label{E:F}
\mc{A}_H:={\mc{M}_H\cap \mc{S}_H}
,\qquad
\mc{N}_H:={\mc{M}_H\cup\mc{S}_H},
\end{equation}
 where we will use $\mc{A}_H$ when considering manifolds with Anosov geodesic flow and $\mc{N}_H$ when considering those with no focal points. 
 
Next, we recall that any manifold with no focal points in which every geodesic encounters a point of negative curvature has Anosov geodesic flow \cite[Corollary 3.4]{Eberlein73}. In particular, the class of manifolds with Anosov geodesic flows includes those with negative curvature. We also recall that a manifold with Anosov geodesic flow does not have conjugate points and for all $\rho \in \SM$
\[
T_\rho(\SM)=E_+(\rho)\oplus E_-(\rho)\oplus \re H_p.
\]
where $E_+,E_-$  are the stable and unstable directions as before. (For other characterizations of manifolds with Anosov geodesic flow, see~\cite[Theorem 3.2]{Eberlein73},~\cite{Eberlein73b}.) Moreover, there exists $C>0$ so that for all $\rho\in \SM$,
\[
 |dG^t(v)|\leq Ce^{\mp \frac{t}{C}}|v|,\qquad{ v\in E_\pm(\rho),\quad t\to  \pm\infty,}
\]
 and the spaces $E_\pm(\rho)$ are H\"older continuous in $\rho$~\cite{Anosov}.

\begin{theorem}\label{T:tangentSpace}
{Let $H\subset M$ be a closed embedded submanifold.\\
If $(M,g)$ has {no focal points}, then
\[
{\sig( \mc{R}_{H} \cap \mc{N}_H)=\sig(\mc{R}_{H} ).}
\]
If  $(M,g)$ has Anosov geodesic flow, then 
\[
\sig( \mc{R}_{H} \cap \mc{A}_H)=\sig(\mc{R}_{H} ).
\]}
\end{theorem}

Theorem \ref{T:tangentSpace} combined with Theorem \ref{thm:recur} give the following result.

\begin{corollary}\label{C:tangent}
{Let $H\subset M$ be a closed embedded submanifold of codimension $k$, and let $A\subset H$ satisfy $\dbox(\partial A)<n-k-\frac{1}{2}$. Then if $(M,g)$ has no focal points and 
\[
\sig(  \mc{N}_H \cap \pi_H^{-1}(A))=0
\]
we have
\begin{equation}
\label{e:teaLeaves}
\int_A\phi_hd\sigma_H=o(h^{\frac{1-k}{2}})
\end{equation}
as $h \to 0^+$ for every sequence $\{\phi_h\}$ of Laplace eigenfunctions. If instead $(M,g)$ has Ansov geodesic flow then~\eqref{e:teaLeaves} holds when
\[
\sig(\mc{A}_H\cap \pi_H^{-1}(A))=0.
\]}
\end{corollary}

\noindent{Note that if $\dim M=2$, then $\mc{N}_H=\mc{A}_H$ since $\dim T_{\rho}(\SNH)=1$. {Indeed}, it is not possible to have both $N_{+}(\rho)\neq\{0\}$ and $N_{-}(\rho)\neq\{0\}$ {unless $N_+(\rho)=N_-(\rho)=T_\rho(\SNH)$} and hence $\mc{M}_H\subset \mc{S}_H$. In~\cite{Wym,Wym2} the author works with $(M,g)$ non-positively curved (and hence having no focal points), $\dim M=2$ and $H=\gamma$ a curve.  He then imposes the condition that for all time $t$ the curvature of $\gamma$, $\kappa_\gamma(t)$, avoids two special values determined by the tangent vector to $\gamma$, $\mathbf{k}_\pm(\gamma'(t))$. He shows that under this condition
\[
\int_\gamma \phi_hd\sigma_\gamma=o(1).
\] 
If $\kappa_\gamma(t)=\mathbf{k}_{\pm}(\gamma'(t))$, then the lift of $\gamma$ to the universal cover of $M$ is tangent to a stable or unstable horosphere at $\gamma(t)$ and $\kappa_\gamma(t)$ is equal to the curvature of that horosphere. Since this implies that $T_{(\gamma(t),\gamma'(t)}SN^*\gamma$ is stable or unstable, the condition there is that $\mc{N}_\gamma=\emptyset.$ Thus, the condition $\sig( \mc{N}_H \cap \pi_H^{-1}(A))=0$ is the generalization to higher codimensions of that in~\cite{Wym,Wym2}. {We note that~\cite{Wym2} obtains the improved upper bound $O(|\log h|^{-\frac{1}{2}})$}.

\subsection{Organization of the paper}

{We divide the paper into two major parts. The first part of the paper contains all of the analysis of solutions to $Pu=o(h)$. The sections in this part,  Section~\ref{S:estimate} and  Section~\ref{S:normal}, contain the proofs of Theorem~\ref{t:local2} and Theorem~\ref{thm:normal} respectively. The second part of our paper, consists of an analysis of the geodesic flow and in particular a study of the recurrent set of $\SNH$. Theorem~\ref{thm:recur} is proved in Section~\ref{S:recur}, and Theorems~\ref{T:applications} and~\ref{T:tangentSpace} are proved in Section~\ref{S:applications}. }

 Note that as already explained, Corollary \ref{C:applications} is an immediate consequence of combining Theorems \ref{thm:recur} and \ref{T:applications}.  
 Also,  Theorem \ref{t:local} is a direct consequence of Theorem~\ref{t:local2}  and Corollary \ref{C:tangent} is a consequence of Theorem \ref{thm:recur} and Theorem \ref{T:tangentSpace}.  
 Finally, Theorem \ref{thm:2d} is exactly part \ref{a4} of Theorem \ref{T:applications}.
\ \\

\noindent {\sc Acknowledgements.} Thanks to Semyon Dyatlov, Patrick Eberlein,  Colin Guillarmou, and Gabriel Paternain for several discussions on hyperbolic dynamics.
J.G. is grateful to the National Science Foundation for support under the Mathematical Sciences Postdoctoral Research Fellowship  DMS-1502661.

\section{Quantitative estimate: Proof of Theorem~\ref{t:local2}}\label{S:estimate}

{In Section \ref{S:A0} we present the ground work needed for the proof of Theorem~\ref{t:local2}. In particular, we state the main technical result, Proposition \ref{P:sloth}, on which the proof of Theorem~\ref{t:local2} hinges.
We then divide the proof of Theorem~\ref{t:local2}  in two parts. 
Assuming the main technical proposition, we first prove the theorem for the case $A=H$ {and $w\in C_c^\infty(H^o)$} in Section~\ref{S:A},  and then generalize it to any subset $A \subset H$ in Section~\ref{S:A2}. Finally, Section~\ref{S:key} is dedicated to the proof of Proposition \ref{P:sloth}.}

Throughout this section we assume that $P$ has {principal} symbol $p$ and $H$ is conormally transverse for $p$ as defined in~\eqref{e:transverse}. 
We also assume throughout this section that $\{\phi_h\}$ is a compactly microlocalized quasimode for $P$. 

\needspace{2cm}
\subsection{Preliminaries.}\label{S:A0}

Let $H \subset M$ be a smooth closed submanifold  and let  $U_H$ be an open neighborhood of $H$ described in local coordinates as $U_H  = \{{(x'',x')}\, {: x\in V\subset \re^d}\},$ where these coordinates are chosen so that $H{\cap U_H}=\{(0,x')\,{: (0,x')\in V}  \}$. 
The coordinates $(x'',x') \in U_H$ induce coordinates $(x'',x', \xi'', \xi')$ on $\Sigma_{U_H}^*M=\{(x,\xi)\in \Sigma_p: \; x \in {U}_H \}$ with $(\xi'',\xi') \in \Sigma_p\cap T^*_{(x'',x')}M$, {and where we continue to write $\Sigma_p=\{p=0\}$.  }
In these coordinates, $\xi'$ is cotangent to $H$ while ${\xi''}$ is conormal to $H$. 
{Since $H$ is conormally transverse for $p$, we may assume, without loss of generality, that    $x''=(x_1,\bar{x})$ with dual coordinates $\xi''=(\xi_1,\bar{\xi}),$ }where  
\[
\partial_{\xi_1}p(x,\xi) \neq 0\text{ on } \{p=0\}\cap N^*\!H.
\]

Consider the cut-off function $\chi_\alpha \in C_c^\infty(\R, [0,1])$ with 

\begin{equation}\label{E:chi}
\chi_\alpha(t) =\begin{cases} 
0 & |t| \geq \alpha \\
1 & |t| \leq \frac{\alpha}{2},
\end{cases}
\end{equation}
with  $|\chi_\alpha'(t)| \leq 3/\alpha$ for all $t \in \R$. For ${\e}>0$ consider the symbol

 \begin{equation}\label{E: beta}
 \beta_\ep(x',\xi')=\chi_\ep( |\xi'|_{g_H(x')}) \in C_c^\infty(T^*H),
 \end{equation} 
 where $g_H$ is the Riemannian metric on $H$ induced by $g$.
Let $w \in C_c^\infty({H^o})$, { where $H^o$ denotes the interior of $H$}. We start splitting the period integral as
\[
\int_H  w \phi_h \, d\sigma_H=\int_H Op_h(\beta_\e)[w \phi_h] \,d\sigma_H+\int_H Op_h(1-\beta_\e)[w \phi_h] \, d\sigma_H.
\]
The same proof as \cite[Lemma 8]{CGT} yields that for all $u \in L_{\comp}^2({H^o})$

\begin{equation*}\label{E:1-beta}
{\left|\int_H Op_h(1-\beta_\e)u\, d\sigma_H\right|
= O_\e(h^\infty)\, \|u\|_{L^2(H)}.}
\end{equation*}
{(see also Lemma~\ref{l:intByParts}).}

Choosing $u=w \phi_h$, and using the restriction bound $\|\phi_h\|_{L^2(H)}=O(h^{-\frac{2k-1}{4}})$ {from} \cite{BGT}, we obtain that 

\begin{equation}\label{E:split}
{\int_H w \phi_h \, d\sigma_H=\int_H Op_h(\beta_\e)[w \phi_h] \,d\sigma_H+O_\e(h^\infty)}.
\end{equation}

We control the integral of $Op_h(\beta_\e )w\phi_h$  using the following lemma. Recall that we write $H_p$ for the Hamiltonian vector field corresponding to $p(x,\xi)$ {and $\varphi_t$ for the associated Hamiltonian flow}. To shorten notation, we write 
\[
\LambdaH:= \bigcup_{|t|\leq T}\varphi_t(\SigH).
\]

\begin{proposition}\label{P:sloth}
Let $\chi\in C_c^\infty(T^*M)$ so that $H_p \chi\equiv 0$ on $\LambdaH$ for some $T>0$. 
Let $w \in C_c^\infty(H)$. There exists {$C_{n,k}=C(n,k)>0$ depending only on $n$ and $k$} so that 
\begin{multline*}
{\lim_{\e\to 0}}\limsup_{h\to 0^+}h^{{k-1}}\left|\int_H Op_h(\beta_\e w)\big [Op_h(\chi) \phi_h\big]d\sigma_H\right|^2
\leq\\
 \! C_{n,k} \, \sig\!\big(\supp(\chi{1_{\SigH}})\big) \int_{\SigH} \!\!w^2\chi^2{|H_pr_H|^{-1}} d\muH.
\end{multline*}
\end{proposition}

The proof of Proposition \ref{P:sloth} is given in Section \ref{S:key}.  The purpose of this proposition is to allow us to use $\chi$ to localize quasimodes to the support of $\lambda_{_{\! H}}$ and its complement.
Since $\lambda_{_{\! H}}$ and $\Sig$ are mutually singular, it is not difficult to see that Proposition \ref{P:sloth} gives a  bound for $\limsup_{h\to 0^+}h^{\frac{k-1}{2}}\left|\int_H w \phi_hd\sigma_H\right|$ of the form $C\left(\int_{\SigH} w^2 f d\Sig\right)^{1/2}$.
 By further restricting $\chi$ to shrinking balls inside $\SigH$ an application of the Lebesgue differentiation theorem allows us to obtain a bound of the form $C \int_{\SigH} |w| \sqrt{f} d\Sig$ as claimed. {This improvement will be needed when passing to subsets $A\subset H$.} The factor $|H_pr_H|^{-1}$ measures the cost of restricting to a hypersurface containing $H$ which is microlocally transversal to $H_p$. In particular, we choose coordinates so that $H\subset \{x_1=0\}$ and $|H_p r_H|=\partial_{\xi_1}p\neq 0$ at a point $\rho\in \SigH$. This is possible since $H$ is conormally transverse for $p$. 

To apply Proposition \ref{P:sloth} it is key to work with cut-off functions $\chi\in C_c^\infty(T^*M)$ so that $H_p \chi\equiv 0$ on $\LambdaH$ for some $T>0$. 
Therefore, the following lemma is dedicated to extending cut-off functions on $\SigH$ to cut-off functions on $T^*M$ that are invariant under the Hamiltonian flow inside $\LambdaH$. Let $T_{_{\! \SigH}}>0$ be so that 
\[
\varphi:~[-2T, 2T] \times \SigH \to \LambdaHtwo
\]
 is a diffeomorphism for all $0 \leq T \leq T_{_{\!\SigH}}$. {Such a $T_{_{\!\SigH}}$ exists since $H$ is compact and conormally transverse for $p$. Moreover, for $T<T_{_{\!\SigH}}$, $\LambdaHtwo$ is a closed embedded submanifold in $T^*M$.}
 
\begin{lemma}\label{L:ant}
For all  $\tilde \chi \in C_c^\infty(\SigH;[0,1])$ and $0\leq T \leq T_{_{\!\SigH}}$ there exists \\$\chi \in C_c^\infty(T^*M;[0,1])$ so that 
\[
\chi (\varphi_t(x,\xi)) =\tilde \chi(x,\xi)
\]
for all $|t| \leq T$ and $(x,\xi) \in \SigH$.  In particular, $H_p \chi\equiv 0$ on $\LambdaH$.
\end{lemma}

\begin{proof}
Let ${\psi} \in C^\infty_c(\R; [0,1])$ be a fixed function supported on $(-2T, 2T)$ with ${\psi}\equiv 1$ on $[-T, T]$.
Then, using that $\varphi:~[-2T, 2T] \times \SigH \to \LambdaHtwo$ is a diffeomorphism,  define the smooth cut-off $\chi: \LambdaHtwo \to [0,1]$ by the relation
\[
\chi (\varphi_t(x,\xi))={\psi}(t) \tilde\chi(x,\xi).
\]
 Finally, extend $\chi$ to all of $T^*M$ so that $\chi\in C_c^\infty(T^*M;[0,1])$. We can make such an extension since $\LambdaH$ is a closed embedded submanifold in $T^*M$.
\end{proof}

\subsection{Proof of Theorem~\ref{t:local2} for $A=H$} \label{S:A}

Fix $\delta>0$. Since $\Sig$ and $\lambda_{_{\! H}}$ are two {Radon} measures on $\SigH$ that are mutually singular,  there exist {$K_\delta\subset \SigH$ compact and $U_\delta \subset \SigH$ with
$K_\delta \subset U_\delta$ and so that 
\[
\Sig (U_\delta) \leq \delta \qquad \text{and} \qquad \lambda_{_{\! H}}(\SigH \backslash K_\delta ) \leq \delta.
\]}
{
Indeed, by definition of mutual singularity, there exist $V,W\subset \SigH$ so that $\lambda_{_{\! H}}(W)=\Sig(V)=0$ and $V\cup W=\SigH$. Hence, by outer regularity of $\Sig$, there exists $ U_\delta\supset V$ open with $\Sig(U_\delta)\leq \delta.$ Next, by inner regularity, of $\lambda_{_{\! H}}$, there exists $K_\delta\subset U_\delta$ compact with $\lambda_{_{\! H}}(\Sig\setminus K_\delta)=\lambda_{_{\! H}}(U_\delta\setminus K_\delta)\leq \delta.$ 
}
Let $\tilde \kappa_\delta \in C^\infty_c(\SigH; [0,1])$ be a cut-off function with 
\[
\tilde \kappa_\delta \equiv 1 \;\; \text{on}\; K_\delta \qquad \quad \text{and} \quad \qquad \supp \tilde \kappa_\delta\subset U_\delta.
\]
Let $\kappa_\delta \in C_c^\infty(T^*M;[0,1])$ be the cut-off extension  {of $\tilde \kappa_\delta $} given in Lemma \ref{L:ant} with
\[
H_p\kappa_\delta \equiv 0 \;\; \text{on}\;\LambdaH,
\]  
where we have fixed  $T>0$ so that $2T \leq T_{_{\SigH}}$.
We use \eqref{E:split} and split the period integral as
\begin{multline*}
\int_H w \phi_h \, d\sigma_H=\int_H Op_h(\beta_\e w)[Op_h(\kappa_\delta) \phi_h] \,d\sigma_H\\+\int_H Op_h(\beta_\e w)[Op_h(1-\kappa_\delta) \phi_h] \,d\sigma_H+O_\e(h^\infty).
\end{multline*}
Applying Proposition \ref{P:sloth}  with  $\chi= \kappa_\delta$,  we have that 

\begin{equation}\label{E:rho}
\begin{aligned}
&{\lim_{\e\to 0}}\limsup_{h\to 0}  h^{k-1} \left|\int_H Op_h(\beta_\ep w) [Op_h(\kappa_\delta) \phi_h]\, d\sigma_H\right|^2\\
&\qquad\qquad\qquad\qquad\qquad\leq  C\, \Sig\!\big(\supp \kappa_\delta {1_ {\SigH}}\big) \int_{\SigH} \kappa_\delta^2 w^2 d\mu_{_{\!H,p}} \leq C\, \delta.
\end{aligned}
\end{equation}
Here we have used that $\Sig (U_\delta) \leq \delta$ and that by construction  $\supp \kappa_\delta {1_{ \SigH}} = \supp \tilde{\kappa}_\delta \subset  U_\delta$.

We dedicate the rest of the proof to showing that 

\begin{align}\label{E:rho2}
\lim_{\ep \to 0} \limsup_{h\to 0} h^{\frac{k-1}{2}} \left|\int_H Op_h(\beta_\ep w)[Op_h(1-\kappa_\delta) \phi_h]d\sigma_H\right|
\leq  C_{n,k} \int_{\SigH} |w| \sqrt{{f_1}} \,d\Sig + C\delta^{\frac{1}{2}}.
\end{align}
where ${f_1:=f{|H_pr_H|^{-1}}}.$
Putting \eqref{E:rho} together with \eqref{E:rho2} then concludes the proof.\\

We start by splitting the left hand side in \eqref{E:rho2}  into an integral over small balls.
By the Besicovitch--Federer Covering Lemma \cite[Theorem 1.14, Example (c)]{Hein01}, there exists a constant $c_{n}>0$ depending only on $n$ and $r_0=r_0(H)$ so that for all $0<r<r_0$, there exist open balls $\{B_1,\dots, B_{N(r)}\} \subset \SigH$ of radius $r$ with 
\[
N(r)\leq c_{n} r^{1-n} \qquad \text{and}\qquad \Sig(B_j) \leq c_{n} r^{n-1},
\]
 so that
\[
\SigH\subset \bigcup_{j=1}^{N(r)}B_j
\]
and each point in $\SigH$ lies in at most $c_{n}$ balls. 
Let $\{\tilde \psi_j\}$ with  $\tilde \psi_j\in C_c^\infty(\SigH;[0,1])$ be a partition of unity associated to $\{B_j\}$, and write $\psi_j$ for the extensions  $ \psi_j\in C_c^\infty(T^*M;[0,1])$ given in Lemma \ref{L:ant} so that $\psi_j (\varphi_t(x,\xi)) =\tilde \psi_j(x,\xi)$ for all $|t| \leq 2T$ and $(x,\xi) \in \SigH$.
With this construction, $H_p\psi_j\equiv 0$ on   $\LambdaHtwo$, 
\[
\sum_{j=1}^{N(r)} \psi_j\equiv 1\;  \text{on}\;  \LambdaHtwo,\qquad \text{and} \qquad \supp (\psi_j {1_{ \SigH}}) \subset B_j.
\]

Let $\Psi:= \sum_{j=1}^{N(r)}\psi_j$. Setting $\chi= (1-\Psi)(1-\kappa_\delta)$ we have $H_p \chi=0$ on $\LambdaH$ and $\supp (\chi {1_{ \SigH}})=\emptyset$ (since $1-\Psi \equiv 0$ on $\Lambda_{H, 2T}$). We then apply Lemma~\ref{P:sloth} to $\chi$, to obtain 
\[
\lim_{\ep\to 0}\limsup_{h\to 0}  h^{\frac{k-1}{2}}\left|\int_H Op_h(\beta_\ep w) [Op_h ((1-\Psi)(1-\kappa_\delta)) \phi_h]d\sigma_H\right|=0.
\]

On the other hand, by the triangle inequality we have
\[
\left|\int_H Op_h(\beta_\ep w) [Op_h (\Psi(1-\kappa_\delta)) \phi_h]d\sigma_H\right| 
\leq    \sum_{j=1}^{N(r)} \left|\int_H Op_h(\beta_\ep w)[Op_h( \psi_j (1-\kappa_\delta)) \phi_h]\, d\sigma_H\right|.
\]
By construction we have that  $H_p[\psi_j (1-\kappa_\delta)] \equiv 0$ on $\LambdaH$. We may therefore apply Proposition \ref{P:sloth}  with $\chi=\psi_j (1-\kappa_\delta)$ to  find that there exist $\ep_0,C_{n,k}>0$ so that 
\begin{multline*}
{\lim_{\e\to 0}}\limsup_{h\to 0} h^{k-1}\left|\int_H Op_h(\beta_\ep w)[Op_h( \psi_j (1-\kappa_\delta)) \phi_h]\, d\sigma_H\right|^2\\
\leq C_{n,k}\,r^{n-1} \int_{\SigH}\psi_j^2 w^2 (1-\kappa_\delta)^2 {|H_pr_H|^{-1}}d\mu_{_{\!H,p}}.
\end{multline*}
Here we have used that $\supp (\psi_j {1_{\SigH}}) \subset B_j$ and { for $r_j>0$ small enough} $\Sig(B_j)\leq c_n r^{n-1}$ for all $j=1, \dots, N(r)$, and some ${c_n}>0$ {depending only on $n$}. It follows that there is $C_{n,k}>0$ for which

\begin{align*}
&\lim_{\ep\to 0}\limsup_{h\to 0}  h^{\frac{k-1}{2}} \left|\int_H Op_h(\beta_\ep w) [Op_h (\Psi(1-\kappa_\delta)) \phi_h]d\sigma_H\right| \leq \\
& \hspace{5cm}\leq  C_{n,k}\,r^{\frac{n-1}{2}}  \sum_{j=1}^{N(r)}\left( \int_{\SigH}\psi_j^2 w^2 (1-\kappa_\delta)^2\, {|H_pr_H|^{-1}}d\mu_{_{\!H,p}}\right)^{\frac{1}{2}}.
\end{align*}

Decomposing $\mu_{_{\!H,p}}= f \Sig +\lambda_{_{\! H}}$, and using that  
\[
\supp ((1-\kappa_\delta){1_ {\SigH }})\, \subset\;   \SigH\backslash K_\delta 
\]
 while $\lambda_{_{\! H}}(\SigH \backslash K_\delta)  \leq \delta$, we conclude that there exists $C>0$ so that 

\begin{align}\label{E:rho3}
\lim_{\ep\to 0}\limsup_{h\to 0}  h^{\frac{k-1}{2}} \left|\int_H Op_h(\beta_\ep w) [Op_h (1-\kappa_\delta) \phi_h]d\sigma_H\right|
&\leq  C_{n,k}F(r)+ C \delta^{1/2}. 
\end{align}
where
\[
F(r):= r^{\frac{n-1}{2}} \sum_{j=1}^{N(r)}\left( \int_{\SigH}\psi_j^2 w^2 {f_1}\,  d\Sig\right)^{\frac{1}{2}}.
\]

{Indeed, applying the triangle inequality, 
\begin{multline*}
C_{n,k}r^{\frac{n-1}{2}}  \sum_{j=1}^{N(r)}\left( \int_{\SigH}\psi_j^2 w^2 (1-\kappa_\delta)^2{|H_pr_H|^{-1}}\, d\mu_{_{\!H,p}}\right)^{\frac{1}{2}}\leq \\
C_{n,k}F(r)+Cr^{\frac{n-1}{2}}\sum_{j=1}^{N(r)}\left(\int_{\SigH}\psi_j^2w^2(1-\kappa_\delta)^2d\lambda_{_{\! H}}\right)^{\frac{1}{2}}.
\end{multline*}
By Cauchy-Schwarz,
\begin{align*}
r^{\frac{n-1}{2}}\sum_{j=1}^{N(r)}\left(\int_{\SigH}\psi_j^2(1-\kappa_\delta)^2w^2d\lambda_{_{\! H}}\right)^{\frac{1}{2}}&\leq r^{\frac{n-1}{2}}(N(r))^{1/2}\left(\int_{\SigH}\sum_{j=1}^{N(r)}\psi_j^2w^2(1-\kappa_\delta)^2d\lambda_{_{\! H}}\right)^{\frac{1}{2}}\\
&\leq C\lambda_{_{\! H}}(\supp(1-\kappa_\delta)1_{\SigH})^{\frac{1}{2}}\\
&\leq C\delta^{1/2},
\end{align*}
{and this proves \eqref{E:rho3}.}
}

Since for $r$ small enough, {and any $j$}, we have $c_{n}^{-1} r^{n-1}\leq {\Sig(B_j)}\leq c_{n} r^{n-1}$, there exists $C_{n,k}>0$ so that

\begin{align*} 
 F(r)&\leq C_{n,k} \int_{\SigH} \left[\sum_{j=1}^{N(r)} \left(\frac{1}{\Sig(B_j)}\int_{B_j}w^2{f_1}\,d\Sig \right)^{\frac{1}{2}}1_{B_j} \right]d\Sig.
\end{align*}

 The Lebesgue Differentiation Theorem \cite[Theorem 3.21]{Folland} shows that

\begin{equation*}
\limsup_{r\to 0} \sum_{j=1}^{N(r)} \left(\frac{1}{\Sig(B_j)}\int_{B_j}w^2 f\, d\Sig \right)^{\frac{1}{2}}1_{B_j} \leq  {C_{n,k}}|w| \sqrt{{f_1}}\quad \;\;\;\;\Sig\!\!-\!\text{a.e.}
\end{equation*}

Furthermore, the weak type 1-1 boundedness of the Hardy--Littlewood maximal function \cite[Theorem 3.17]{Folland} implies that there exists $C_0$ so that for every $\alpha>0$
\[
\Sig\left((x,\xi)\in \SigH:\; \sup_{r>0}\left(\frac{1}{\Sig(B_j)}\int_{B_j} w^2{f_1}d\Sig\right)^{\frac{1}{2}}\geq \alpha \right)\leq C_0\alpha^{-2}.
\]

Hence, by the Dominated Convergence Theorem, 
\begin{equation}\label{E:rho4}
F(r)\leq C_{n,k}\int_{\SigH} |w|\sqrt{{f_1}}\,d\Sig.
\end{equation}
Feeding \eqref{E:rho4} into \eqref{E:rho3} proves \eqref{E:rho2}. Putting \eqref{E:rho} together with \eqref{E:rho2}  concludes the proof.
\qed \ \\

\subsection{Proof of Theorem~\ref{t:local2} for any $A\subset H$} \label{S:A2}

{In order to pass to $A\subset H$, we break the integral into two pieces. First, near the conormal bundle $N^*\!H$, we approximate $1_A$ by an ($h$-independent) smooth function and apply the theorem on all of $H$. In order to estimate the piece away from $N^*\!H$, we approximate $1_A$ by a smooth function depending badly on $h$. We are then able to perform integration by parts to estimate contributions away from $\partial A$ and a simple volume bound near $\partial A$.  }

 Let $A \subset H$ be a subset with  $\dbox(\partial A)<n-k -\frac{1}{2}$ and indicator function $1_{A}.$  {Extend $H$ to $\tilde H$} another closed, embedded submanifold of codimension $k$ so that $H$ is compactly contained in the interior $\tilde H^o$. We will actually apply Theorem~\ref{t:local2}  to $\tilde H$ and $w\in C_c^\infty(\tilde H^o)$.
 Since $C_{{c}}^{\infty}(\tilde H^o)$ is dense in $L_{\comp}^2(\tilde H^o)$, for any $\delta>0$, we can find a positive function $\psi_A \in C_c^{\infty}(\tilde H^o)$ with
$$
 \| \psi_A - 1_{A} \|_{L^2(\tilde H)} \leq \delta.
 $$
For any $\ep>0$ and $w \in C^\infty_c(\tilde H^o)$,
\[
\Big| \int_{\tilde H} 1_A \, w\, \phi_h d\sigma_{\tilde H}  \Big|\leq  \Big| \int_{\tilde H} 1_AOp_h(\beta_\ep) (w\, \phi_h) d\sigma_{\tilde H} \Big| + \Big| \langle(1-Op_h(\beta_\ep)) (w \phi_h) ,1_A\rangle_{\tilde H} \Big|.
 \]
 
 We claim that if  $A\subset H$ has boundary satisfying $\dbox(\partial A)<n-k-\frac{1}{2}$
 Then, for all $\delta>0$ and $\epsilon>0$,
 
\begin{equation}\label{E:warm}
\|(1-Op_h(\beta_\ep))^* 1_A\|_{L^2(\tilde H)}=O_{\e{,\delta}}(h^{\frac{1}{4}+\delta}).
\end{equation}
We postpone the proof of \eqref{E:warm} until the end. Assuming that \eqref{E:warm}  holds,  the universal upper bound $\|\phi_h\|_{L^2(\tilde H)}\leq Ch^{-\frac{k}{2}+\frac{1}{4}}$ \cite{BGT} together with Cauchy-Schwarz give

 \begin{align}
& h^{\frac{k-1}{2}}\Big| \int_{H} 1_A \, w\, \phi_h d\sigma_H  \Big| \leq \notag\\
& \leq h^{\frac{k-1}{2}}\Big| \int_{\tilde H} 1_AOp_h(\beta_\ep) (w\, \phi_h) d\sigma_H \Big| + o_\e(1) \nonumber\\
& \leq h^{\frac{k-1}{2}}\Big| \int_{\tilde H} (1_A- \psi_A) Op_h(\beta_\ep) (w\phi_h) d\sigma_{\tilde H} \Big|+ h^{\frac{k-1}{2}}\Big| \int_{\tilde H} \psi_A Op_h(\beta_\ep) (w\phi_h) d\sigma_{\tilde H} \Big|+o_\e(1)\nonumber  \\
&=:T_{1,h}+T_{2,h}+o_\e(1). \label{e:octopus}
\end{align}

Next, note that $\|Op_h(\beta_\ep)(w \, \phi_h)\|_{L^2(\tilde H)}=O({h^{\frac{1-k}{2}}})$ and apply Cauchy--Schwarz to obtain
\[
T_{1,h} \leq \|1_A-\psi_A\|_{L^2(\tilde H)}h^{\frac{k-1}{2}}\|Op_h(\beta_\ep)(w\, \phi_h)\|_{L^2(\tilde H)}\leq C\delta,
\]
for some $C>0$.
Finally, to bound the second term in \eqref{e:octopus} we note that
\[
T_{2,h}=h^{\frac{k-1}{2}}\left|\int_{\tilde H}  Op_h(\beta_\ep) (\psi_A \, w\, \phi_h) d\sigma_{\tilde H}\right|+o(1)=h^{\frac{k-1}{2}}\left|\int_{\tilde H}  \psi_A \, w\, \phi_h \, d\sigma_{\tilde H}\right| + o(1),
\]
and that by Theorem \ref{t:local2} with $A=\tilde H$ {and $w\in C_c^\infty(\tilde H^o)$} there exists $C_{n,k}>0$ for which 
\begin{multline*}
\limsup_{h\to 0}T_{2,h} \leq C_{n,k}\int_{\Sigma_{\tilde H}} \psi_A \, |w|  \sqrt{f|H_pr_H|^{-1}} \, d{\sigma_{\!_{S\!N^*\!\tilde H}}}\leq\\
C_{n,k} \int_{\pi_H^{-1}(A)}  |w| \sqrt{f|H_pr_H|^{-1}} \, d\Sig +C\delta  \|f\|_{L^1(\tilde H)} \|w\|_{L^{{\infty}}(\tilde H)}.
\end{multline*}
The last equality follows from Cauchy-Schwarz and the bound $\| \psi_A - 1_{A} \|_{L^2(\tilde H)} \leq \delta$. This gives the stated result provided \eqref{E:warm} holds. We proceed to prove \eqref{E:warm}.\\

{To prove \eqref{E:warm} we first introduce a cut-off function $\chi_h \in C_c^\infty(\tilde H^o)$ so that  $(1-\chi_h)1_A$  is smooth and close to $1_A$ {and $\chi_h $ is $1$ in a neighborhood of $\partial A$.}
For this, fix ${0<\delta}<1$ and cover $\partial A$ by {\small{$(n-k)$}-dimensional cubes} $Q_{i,h} \subset \tilde H^o$,  with $1\leq i\leq N(h)$,  and  side length $h^{\delta}$ with disjoint interiors. This can by done so that 
$$
\limsup_{h\to 0^+}\frac{\log N(h)}{\delta \log h^{-1}}=\dbox(\partial A).
$$
{We decompose
\begin{multline}\label{E:warm1}
\|(1-Op_h(\beta_\ep))^* 1_A\|_{L^2(\tilde H)}=\|(1-Op_h(\beta_\ep))^* (1-\chi_h)1_A\|_{L^2(\tilde H)}\\+\|(1-Op_h(\beta_\ep))^*\chi_h 1_A\|_{L^2(\tilde H)}.
\end{multline}
We bound $\|(1-Op_h(\beta_\ep))^*\chi_h 1_A\|_{L^2(\tilde H)}$ using that $1-Op_h(\beta_\ep)$ is $L^2$-bounded and that $\chi_h 1_A$ has compact support. We proceed to bound $\|\chi_h1_A\|_{L^2(\tilde H)}$.
}
Cover each cube {$Q_{i,h}$} by $2^{n-k}$ open balls {$B_{i,h}$} of radius $h^\delta$. Let $ \chi_{i,h}\in C_c^\infty(B_{i,h};[0,1])$ be a partition of unity near $\partial A$ subordinate to $B_{i,h}$ and define ${\chi_h=\sum_{i=1}^{N(h)}\chi_{i,h}}$.
Then, 
\begin{gather}
\chi_h \equiv 1\text{ in a neighborhood of }\partial A, \qquad \supp \chi_h \subset \{x\in H:\; d(x,\partial A)\leq 2h^{\delta}\}, \nonumber  \\
|\partial_x^\alpha \chi_h |=O_\alpha (h^{-|\alpha|\delta}).\label{e:anteater}
\end{gather}
Moreover, {since the volume of each cube $Q_{i,h}$ is $h^{\delta(n-k)}$, there is $C>0$ so that }
\begin{equation*}
\|\chi_h\|^2_{L^2(\tilde H)}\leq C N(h)h^{\delta(n-k)}\leq Ch^{\delta(n-k-\dbox(\partial A))}.
\end{equation*}
{It follows that 
\begin{equation}
\label{e:nearEst}
\|(1-Op_h(\beta_\ep))^*\chi_h 1_A\|_{L^2(\tilde H)}=O\Big(h^{\frac{\delta}{2}(n-k-\dbox(\partial A))}\Big).
\end{equation}}

On the other hand, the function  $(1-\chi_h)1_A$ satisfies the bounds~\eqref{e:anteater}. In particular, putting $\psi_h=1-\chi_h$ in Lemma \ref{l:intByParts} below, for $\delta<1$,
\begin{equation}
\label{e:awayEst}
\|(1-Op_h(\beta_\ep))^*(1-\chi_h)1_A\|_{{L^\infty(\tilde H)}}=O_{\epsilon}(h^\infty).
\end{equation}
}

\noindent {Combining \eqref{e:nearEst} and \eqref{e:awayEst}  into \eqref{E:warm1}, and taking $0<\delta<1$ sufficiently close to $1$,  proves \eqref{E:warm} as claimed.}

\qed

\begin{lemma}
\label{l:intByParts}
Suppose that $\psi_h\in C_c^\infty(\tilde H^o)$ satisfies~\eqref{e:anteater} for some $0<\delta<1$. Then,  for $u\in L^2(\tilde H)$,
$$
\|(1-Op_h(\beta_\e))^*{(\psi_h u)}\|_{L^\infty(\tilde H)}=O_{\e,\delta}(h^\infty{\|u\|_{L^2(\tilde{H})}}).
$$
\end{lemma}

\begin{proof}
Integrating by parts with 
\[
L:= \frac{1}{|x-x'|^2 + |\xi'|^2} \left( \sum_{j=1}^n \xi_j'  hD_{x'_j} +  \sum_{j=1}^n (x'_j-x_j) hD_{\xi'_j} \right),
\]
 {relation \eqref{e:anteater}} gives
\begin{align*}
&[(1-Op_h(\beta_\ep))^*\psi_h{u}](x)=\\
&\qquad=\frac{1}{(2\pi h)^{n-k}}\iint e^{\frac{i}{h}\langle x-x',\xi'\rangle} (1-\beta_\ep(x',\xi'))(\psi_h(x')){u}(x')dx'd\xi'\\
&\qquad= \frac{1}{(2\pi h)^{n-k}}\iint e^{\frac{i}{h}\langle x-x',\xi'\rangle}(L^*)^N\big[(1-\beta_\ep(x',\xi'))\psi_h(x'){u}(x')\big]dx'd\xi'\\
&\qquad=O_{\e,N}(h^{k-n+N(1-\delta)}{\|u\|_{L^2(\tilde{H})}}).
\end{align*}
\end{proof}

\subsection{Localizing near bicharacteristics: Proof of Proposition \ref{P:sloth}}\label{S:key}

{Throughout the proof of Proposition \ref{P:sloth} we will need the following lemma.
 Since it is a local result, we state it for functions and operators acting on $\R^n$. 
We write $(x_1, \tilde{x}) \in \R \times \R^{n-1}$ for coordinates in $\R^n$ and $(\xi_1, \tilde{\xi})$ for the dual coordinates.

{\begin{lemma}
\label{l:squirrel}
Let $\kappa=\kappa(x_1, \tilde{x},\tilde{\xi})$ be a smooth function with compact support and fix
 $\rho_0\in T^*\R^n$ with 
\[
p(\rho_0)\neq 0 \quad \text{or} \quad  \partial_{\xi_1}p(\rho_0)\neq 0.
\]
 Then, {there exists $C_0,T_0>0$ and a neighborhood $V$ of $\rho_0$ so that for all $0<T<T_0$ the following holds.} Let $U$ be a neighborhood of $\supp \kappa$ and $b\in C_c^\infty(T^*\R^n)$ with 
\[
\bigcup_{|t|<T} \varphi_t( \{p=0\} \cap U) \subset \{ b\equiv 1\}.
\]
Let $\chi\in C_c^\infty(V)$, $\tilde{\chi}\in C_c^\infty(T^*\R^n)$ with $\tilde{\chi}\equiv 1$ in a neighborhood of $\supp \chi$, and $q=q(x_1)\in C^\infty(\R; S^\infty(T^*\R^{n-1}))$. {Then, there exists $C>0$ so that the following hold.} \\  \ \smallskip
 If $p(\rho_0)\neq 0$, {then  }
$$
\|{Op_h(q)}Op_h(\kappa )Op_h(\chi)\phi_h (0, \cdot) \|_{L^2_{ \tilde{x}}}\leq C\|Op_h(\tilde{\chi})P{\phi_h}\|_{{L^2_x}}.
$$
If $p(\rho_0)=0$, then
\begin{align*} 
\label{UPSHOT 1}
&\|{Op_h(q)}Op_h(\kappa )Op_h(\chi)\phi_h (0, \cdot) \|_{L^2_{ \tilde{x}}} \leq \\
&\qquad\qquad\qquad\qquad  4T^{-\frac{1}{2}} |\partial_{\xi_1}p(\rho_0)|^{-\frac{1}{2}}\|Op_h(b){ Op_h( \chi)Op_h( q) } \phi_h\|_{L^2_{x}}\\
&\qquad\qquad\qquad\qquad+C_0T^{\frac{1}{2}}h^{-1}\| Op_h(b) Op_h(p){Op_h(\chi)Op_h(q)}\phi_h\|_{L^2_{x}}+Ch^{-1}\|Op_h(\tilde{\chi})P{\phi_h}\|\\
&\qquad\qquad\qquad\qquad+Ch^{1/2}\|Op_h(\tilde{\chi}){\phi_h}\|_{L^2_x}+O(h^\infty)\|{\phi_h}\|_{L^2_x}.
\end{align*}
\end{lemma}}


\noindent The proof of Lemma  \ref{l:squirrel} is very similar to that of \cite[Lemma 4.3]{Gdefect}, although some alterations are needed.  For the sake of completeness we include the proof at the end of this section, in \ref{S: squirrel proof}.
}

{
 \subsubsection{\bf Case: $H$ is  hypersurface.} We proceed to explain the role that  Lemma  \ref{l:squirrel} has in the proof of  Proposition \ref{P:sloth}. To do this, we assume for a moment that 
 $H$ is a hypersurface ($k=1)$, and use local coordinates near it $(x_1, x')$ with  $H=\{x_1=0\}$.  This section is a particular case of the results presented in Section \ref{S: any k} where $H$ with any codimension $k$ is treated.  
 
 Let   $w\in C_c^\infty(H^o)$, and  let  ${\we}\in C_c^\infty(H)$ with 
{\[\we(x') \equiv 1  \; \;\;\text{for }\; \;x'\in \supp  w, \qquad\qquad \qquad  \lim_{\e \to 0}\we = 1_{\supp w}.\] } Define  
\[
\kappa_\e(x,\xi)=\chi_\e(|x_1|)\beta_\e(x',\xi')\we(x').
\] 
 Also, let $\chi, \tilde{w} \in C_c^\infty(T^*M)$ supported sufficiently close to $\rho_0\in \SigH$ satisfy
$$H_p\chi \equiv 0,\,\text{ on }\LambdaH,\qquad H_p\tilde{w}\equiv 0\text{ on }\LambdaH,\, \tilde{w}|_{\SigH}=w$$
where $0<T\leq T_\chi$ and $T_\chi$ is defined in \eqref{E:T}.

{We choose Fermi coordinates with respect to $H$ so that
\[
|H_pr_H(\rho_0)|=\partial_{\xi_1}p(\rho_0)\neq0\qquad\text{or}\qquad p(\rho_0)=0.
\]
Moreover, in these coordinates $\|u\|_{L^2_x}\leq 2\|u\|_{L^2(M)}$. 
} 
Hence, we will apply Lemma~\ref{l:squirrel} with $\kappa=\kappa_\e$ and { $\chi$ (here we shrink the support of $\chi$ if necessary)}. In order to apply the lemma,  we note that 
\[
\supp \kappa_\e \cap \{p=0\} \subset    \{(x,\xi):\;  |x_1|\leq 3\e,\, |\xi'|\leq 3\e,\, p=0 \},
\]
and define $b_{\e}\in C_c^\infty(T^*M; [0,1])$ so that
\begin{equation}
\label{e:platypus}
\begin{gathered}
\bullet \; b_\ep\equiv 1\;\;\; \text{on}\;\; \bigcup_{|t|\leq T/3} \varphi_t(\{(x,\xi):\;  |x_1|\leq 3\e,\, |\xi'|\leq 3\e,\, p=0 \}),\\
\bullet \, \supp b_\e \; \subset \; \bigcup_{|t|\leq T/2}\varphi_t(\{(x,\xi):\; |x_1|\leq 4\e,\; |\xi'|\leq 4\e, |p|\leq 2\e \}).
\end{gathered}
\end{equation}

Next, let $\twe$ be an extension of $\we$ off of $\SigH$ so that $H_p \twe\equiv 0$ in a neighborhood of $b_\ep \equiv 1$.
{ Applying Lemma \ref{l:squirrel} with $\kappa=\kappa_\e$, $\chi$, $b=b_\e\twe$, and $q=1$, gives the existence of {$C_0>0$ independent of $T$} so that 
\begin{align*}
&\|Op_h(\kappa_\e)Op_h(\tilde{w}\chi) \phi_{h}\|_{L^2(H)}\leq 8T^{-\frac{1}{2}}|\partial_{\xi_1}p(\rho_0)|^{-\frac{1}{2}}\|Op_h(b_{\e}\twe)Op_h(\tilde{w}\chi )\phi_{h}\|_{L^2(M)}\\
&\qquad\quad\qquad\qquad\qquad\qquad
+C_0T^{\frac{1}{2}}h^{-1}\|Op_h(b_\e\twe)POp_h(\tilde{w}\chi ) \phi_{h}\|_{L^2(M)}+o_{\e}(1).
\end{align*}}

Next, we use that  $H_p(\tilde{w}\chi)=0$, $P\phi_h={o(h)}$, and  
\begin{equation}\label{E:bedbug}
{P}Op_h({\tilde{w}}\chi) \phi_{h}=Op_h({\tilde{w}}\chi) {P} \phi_{h}+ \frac{h}{i}Op_h(H_p({\tilde{w}}\chi)) \phi_{h} + O_{L^2}(h^2).
\end{equation}
In addition,  by~\eqref{e:platypus} and the fact that $0\leq b_\e^2\leq 1$, we have 
\[
\lim_{\e\to 0} b_\e^2\twe^2\leq 1_{\LambdaH}{1_{\supp \tilde{w}}}.
\]
 Therefore,
\begin{align}
\lim_{\ep \to 0} \limsup_{h\to 0^+} \|Op_h(\beta_\e)Op_h(\chi w) \phi_{h}\|^2_{L^2(H)}
&\leq \lim_{\ep \to 0} \limsup_{h\to 0^+} \|Op_h(\kappa_\e)Op_h({\tilde{w}}\chi) \phi_{h}\|_{L^2(H)}^2\\
&\leq  {128T^{-1}|\partial_{\xi_1}p(\rho_0)|^{-1}} \int_{\LambdaH}  \tilde{w}^2 \chi^2 d\mu \label{e:likelemma}.
\end{align}

We show in Section \ref{S: sloth proof} how to rewrite the $d\mu$ integral in terms of {an integral with respect to} $d\muH$ to get
\begin{multline*}
{\lim_{\e\to 0}}\limsup_{h\to 0^+}\left|\int_H Op_h(\beta_\e w)\big [Op_h({\tilde{w}}\chi ) \phi_h\big]d\sigma_H\right|^2\\
\leq \! {C_{n,k}|\partial_{\xi_1}p(\rho_0)|^{-1}} \, \sig\!\big(\supp(\chi{1_{\SigH}})\big) \int_{\SigH} \!\!w^2\chi^2 d\muH.
\end{multline*}
as claimed in Proposition \ref{P:sloth}.
}
 \subsubsection{\bf Case: $H$ has any codimension $k$.}\label{S: any k}   {In the case in which $H$ has any codimension $k$, the proof  of Proposition \ref{P:sloth}  hinges on Lemma \ref{L:optimized} below.}  This lemma is dedicated to obtaining a gain in the bound for $\|Op_h(\beta_\e )Op_h(\chi)\phi_h\|_{L^2(H)}$ by localizing in phase space near bicharacteristics emanating from $\SigH$. {The key idea is that microlocalization near a family of bicharacteristics parametrized by $H$ implies a quantitative gain in the $L^2(H)$ norm. By decomposing $\phi_h$ into many pieces microlocalized along well-chosen families of bicharacteristics, we are able to extract Proposition~\ref{P:sloth}.}

Let $\Xi: H\to \SigH$ be a smooth section (i.e. $\Xi\in C^\infty$ and $\Xi(x)\in T^*_xM$); where we continue to write $\SigH=\{p=0\}\cap N^*\!H$. 
\begin{figure}[h!]
\includegraphics[width=6cm]{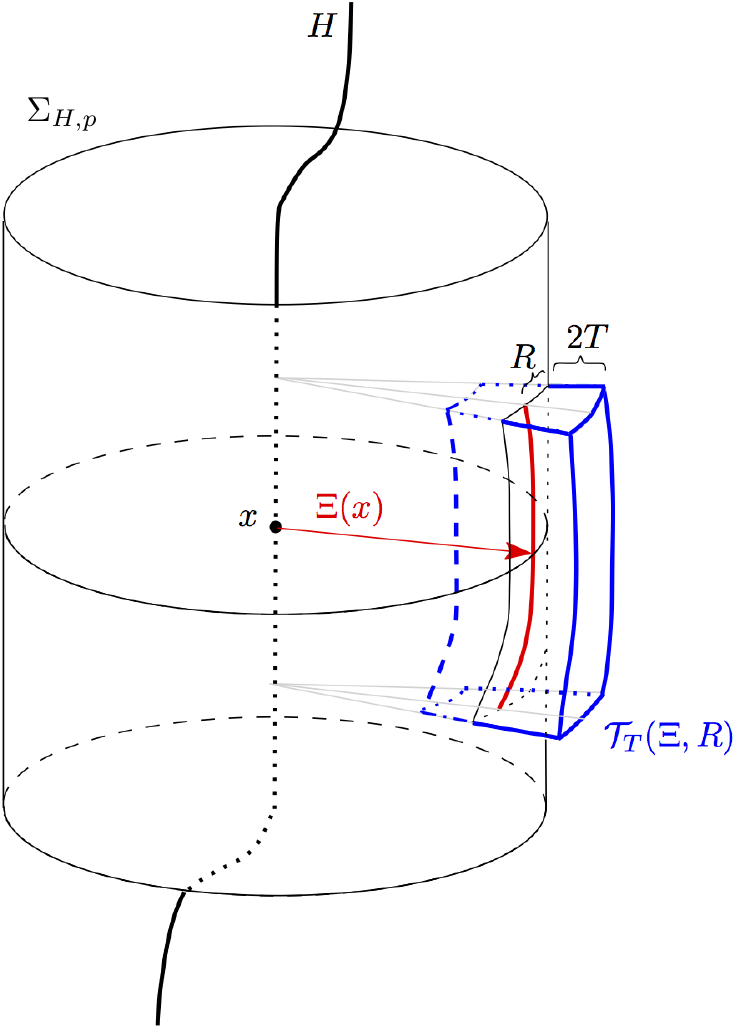}
\caption{\label{f:schem}We show a schematic of $\Xi(x)$, $\SigH$, and $\mc{T}_{_T}(\Xi,R)$ for $H$ a curve and $d=3$. }
\end{figure}
Let $\chi \in C_c^\infty(T^*M)$ {supported near $\rho_0\in \SigH$}. We choose Fermi coordinates with respect to $H$, $(x_1,\bar{x},x')$, so that $H=\{(x_1,\bar{x})=0\}$ and, making additional rotation in $(x_1,\bar{x})$ if necessary, so that 
$$
|H_pr_H(\rho_0)|=\partial_{\xi_1}p(\rho_0)\neq 0.
$$
Moreover, note that for $u$ supported near $x_0$ we have $\|u\|_{L^2_x}\leq 2\|u\|_{L^2(M)}$. 

For each  $(0,x')\in H$  in the projection of $\supp \chi$ onto $H$ define a function $a(x_1;x')$ so that $\xi-a(x_1;x')$ vanishes on the bicharacteristic emanating from $((0,x'),\Xi((0,x'))))$. This is possible since we have chosen coordinates so that 
$$
\partial_{\xi_1}p(\rho_0)\neq0,
$$
and hence the bicharacteristic emanating from $((0,x'),\Xi((0,x'))))$ may be written locally as 
\begin{equation}\label{E:T}
\gamma_{x'}: (-T_\chi,T_\chi) \to T^*M, \qquad \qquad  \gamma_{x'}(x_1)=(x(x_1;x'),a(x_1;x') )
\end{equation}
 where $T_\chi>0$ is small enough, and $x,a$ are smooth functions depending on $\chi$.
{Indeed, if we write $\gamma_{x'}(t)=(x(t), \xi(t))$, we have that $\frac{d}{dt}x_1(t)=\partial_{\xi_1}p(\gamma_{x'}(t))$ which allows us to use the inverse function theorem to locally write $t=t(x_1)$ as a function of $x_1$.}

 To exploit the construction of the function $a$ we further localize in phase space on tubes of small radius $R$ that cover $\supp (\chi {1_{\SigH}})$. 
 We define the tubes 
\begin{equation}
\label{e:tubes}
\mc{T}_{_T}(\Xi,R):= \bigcup_{|t| \leq 2T } \varphi_t(\{(x,\xi)\in \SigH:\;  d((x,\xi),(x,\Xi(x)))<R\}),
\end{equation}
where  $d((x,\xi),(x,\Xi(x)))$ describes the distance in $\SigH\cap T_x^*M$ between the points  $(x,\xi)$ and $(x,\Xi(x))$ (see Figure~\ref{f:schem} for a schematic picture of these objects).

The spirit of the following result  is similar to that of \cite[Lemma 5.2]{Gdefect}. 
Lemma \ref{L:optimized} is dedicated to showing that microlocalizing with $\chi$  supported on $\mc{T}_T(\Xi, R)$ gives an $R^{k-1}$ gain in the bound for $\|Op_h(\beta_\e w)Op_h(\chi )\phi_h\|_{L^2(H)}$.  {This is a generalization of the relation \eqref{e:likelemma} already discussed in the case in which $H$ is a hypersurface.}

\needspace{2in}
\begin{lemma}\label{L:optimized}
Let $\chi \in C_c^\infty(T^*M)$ {supported sufficiently close to $\rho_0\in \SigH$} satisfy
$$H_p\chi \equiv 0,\,\text{ on }\LambdaH,$$
where $0<T\leq T_\chi$ and $T_\chi$ is defined in \eqref{E:T}.
Let $\Xi: H\to \SigH$ be a smooth section. {There exists $C>0$ depending only on $(M,g,H)$ so that for all $R>0$ and $w\in C_c^\infty(H^o)$} if
\begin{equation}\label{E:support}
\supp (\chi {1_{\LambdaH}})\subset  \mc{T}_{_T}(\Xi,R),
\end{equation}
then {there exists $C_{n,k}>0$ depending only on $n$ and $k$ so that}
\begin{equation}
\label{e:optimized}
\lim_{\e \to 0}\limsup_{h\to 0}h^{k-1}\|Op_h(\beta_\e w)Op_h(\chi )\phi_h\|_{L^2(H)}^2\leq \, {C_{n,k}}\,\frac{R^{k-1}}{T|H_pr_H(\rho_0)|} \int_{\LambdaH}\chi^2 {\tilde{w}}^2d\mu,
\end{equation}
where $\tilde{w}\in C^\infty_c(T^*M)$ is any extension of $w$ for which $H_p \tilde w\equiv 0$ on $\LambdaH$.
In addition, if the assumption in  \eqref{E:support} is not enforced, then \eqref{e:optimized} holds with $R=1$.
\end{lemma}

\begin{proof}
In what follows we write $\bar{x}$ for the normal coordinates to $H$ that are not $x_1$. With this notation $x=(x_1, \bar{x}, x')$.
As before, let  ${\we}\in C_c^\infty(H)$ with 
{\[\we(x') \equiv 1  \; \;\;\text{for }\; \;x'\in \supp  w, \qquad\qquad \qquad  \lim_{\e \to 0}\we = 1_{\supp w}.\] }
Define also 
\[\kappa_\e(x,\xi)=\beta_\e(x',\xi')\chi_\e(|(x_1,\bar{x})|)\we(x').\] 
and $\tilde{w}\in C_c^\infty(T^*M)$ with 
$$
H_p\tilde{w}=0,\quad \text{ on }\LambdaH,\qquad \tilde{w}|_{\SigH}=w.
$$
 {Using that $\|\phi_h\|_{L^2(H)}\leq Ch^{-\frac{k}{2}}$}, we bound 
\[
\|Op_h(\beta_\e w)Op_h(\chi )\phi_{h}\|_{L^2(H)}\leq\|Op_h(\kappa_{\e}{\tilde{w}}\chi )\phi_{h}\|_{L^2(H)}{+O_\e(h^{\frac{2-k}{2}})}=\|v_h\|_{L^2(H)}+{o_{\e}(h^{\frac{1-k}{2}}).}
\]
 for
\[
v_h:=e^{-\frac{i}{h}\langle \bar{x}\,,\,\bar{a}(x_1;x')\rangle}Op_h(\kappa_{\e} {\tilde{w}}\chi) \phi_h ,
\]
 \medskip
where $\bar{a}(x_1;x')=(a_{2}(x_1,x'),\dots,a_{k}(x_1,x'))$ and $a$ is defined in \eqref{E:T}. The reason for working with this function $v_h$ is that 
\[
(hD_{x_i})^\ell  v_h=(hD_{x_i}-a_{i})^\ell (Op_h(\kappa_\ep{\tilde{w}} \chi ) \phi_h),
\]
for $i=2,\dots,k$, and this will allow to obtain a gain in the $L^2$-norm bound, {since, as we will see below, $\sup_{T_\delta(\Xi,R)\cap \LambdaH} \max_i |\xi_i-a_{i}(x_1,x')| \leq 3R$.} 
We bound $\|v_h\|_{L^2(H)}$ using  {the version of the Sobolev Embedding Theorem given in} \cite[Lemma 5.1]{Gdefect} which states that if $\ell > (k-1)/2$, then for all $\alpha>0$ {there exists $C_{\ell,k}>0$ depending only on $\ell$ and $k$ so that }

\[
\|v_h(x_1,\cdot,x')\|_{L^\infty_{\bar{x}}}\!\leq \!C_{\ell,k}  h^{1-k} \!\!\left( \!\!\alpha^{k-1}\|v_h(x_1,\cdot,x')\|^2_{L_{\bar{x}}^2}+ \alpha^{k-1-2\ell} {\sum_{i=2}^{k}}\|(hD_{x_i})^\ell v_h(x_1,\cdot,x')\|^2_{L_{\bar{x}}^2}\!\right)\!,
\]
for all $x_1$, $x'$. Now, for all $x_1, \bar{x}$,  integrate in $x'$ to get
\[
\|v_h(x_1,\bar{x},\cdot)\|^2_{L^2_{x'}}\leq C_{\ell,k}  h^{1-k} \left( \alpha^{k-1}\|v_h(x_1,\cdot)\|^2_{L_{\bar{x},x'}^2}+ \alpha^{k-1-2\ell} {\sum_{i=2}^{k}}\|(hD_{x_i})^\ell v_h(x_1,\cdot)\|^2_{L_{\bar{x},x'}^2}\right).
\]
In particular, setting $(x_1, \bar{x})=(0,0)$ on the left hand side we get 
\begin{equation}
\label{e:centipede}
\|v_h\|^2_{L^2(H)}\leq C_{\ell,k}  h^{1-k} \left( \alpha^{k-1}\|v_h(0,\cdot)\|^2_{L_{\bar{x},x'}^2}+ \alpha^{k-1-2\ell} {\sum_{i=2}^{k}}\|(hD_{x_i})^\ell v_h(0,\cdot)\|^2_{L_{\bar{x},x'}^2}\right).
\end{equation}
{We will end up choosing $\alpha=R$ and $\ell=k$.}

\begin{remark}
{Note that when $k=1$ (i.e. in the case of $H$ is a hypersurface), estimates on the derivatives are not necessary.}
\end{remark}

By~\eqref{e:transverse} we may assume, without loss of generality, that $\partial_{\xi_1}p\neq0$ on $\supp \kappa_\e\cap \{p=0\}$. Hence, we will apply Lemma~\ref{l:squirrel} with $\kappa=\kappa_\e$ and {$\chi$ (here we shrink the support of $\chi$ if necessary)}.  {In order to apply the lemma we define $b_{\e}\in C_c^\infty(T^*M; [0,1])$  as in \eqref{e:platypus}, where we change $|x_1|$ for $|(x_1, \bar x)|$.} Next, let $ \twe$ be an extension of $\we$ off of $\SigH$ so that $H_p\we\equiv 0$ in a neighborhood of $b_\ep \equiv 1$.
 We do this as in Lemma~\ref{L:ant} using that $H_p$ is transverse to $\SigH$ to solve the initial value problem.

We now choose $q$ to obtain a gain in the $L^2(H)$ restriction norm related to $R$.
{ Let 
$$ 
T_{\rho_0}:=T|\partial_{\xi_1}p(\rho_0)|.
$$}
 {Applying Lemma~\ref{l:squirrel} with $\kappa=\kappa_\e$, $\chi$, $b=\twe b_\e$, and $q=1$, we have }
\begin{multline*}
{\|v_h(0,\cdot)\|}_{L_{\bar{x},x'}^2}\leq  {8T_{\rho_0}^{-\frac{1}{2}}}\|Op_h(\twe b_\e{\tilde{w}}\chi )\phi_{h}\|_{L^2(M)}
\\+{C_0T^{\frac{1}{2}}}h^{-1}\|Op_h( \twe b_\e)POp_h({\tilde{w}}\chi ) \phi_{h}\|_{L^2(M)}+o_{\e}(1)
\end{multline*}
{with $C_0>0$ independent of $T$. Here we have used that in our coordinates $\|u\|_{L_x^2}\leq 2\|u\|_{L^2(M)}.$}

Let $\ell$ with $2\ell>k-1$ and define {
\[
Q_i=(hD_{x_i}-a_i)^\ell\qquad \text{and}\qquad Q_i=Op_h(q_i).
\]}  
In particular,  $q_i=(\xi_i-a_i)^\ell+O(h)$. Then,  Lemma \ref{l:squirrel} gives that there exists {$C_0>0$ independent of $T$} so that
\begin{multline*} 
{\|(hD_{x_i})^\ell v_h(0,\cdot)\|}_{L_{\bar{x},x'}^2}\leq {128T_{\rho_0}^{-\frac{1}{2}}}
\|Op(\twe b_\e)Op_h({\tilde{w}}\chi ){Q_i}\phi_{h}\|_{L^2(M)}\\
\qquad \qquad+{C_0T^{\frac{1}{2}}}h^{-1}\|Op_h(\twe b_\e)POp({\tilde{w}}\chi ){Q_i}\phi_{h}\|_{L^2(M)}+o_{\e}(1).
\end{multline*}

Applying~\eqref{e:centipede}  {gives that for any $\alpha>0$}
\begin{align}
&h^{k-1}\|Op_h(\beta_\e w)Op_h(\chi ) \phi_{h}\|^2_{L^2(H)} \notag \\
&\leq C_{\ell,k}\alpha^{k-1}\left({T_{\rho_0}^{-1}}\|Op_h(\twe b_\e{\tilde{w}}\chi)\phi_h\|_{L^2(M)}^2+ h^{-2}{C_0^2T}\|Op_h(\twe b_\e)POp_h({\tilde{w}}\chi )\phi_{h}\|_{L^2(M)}^2\right) \notag\\
&\quad+C_{\ell,k}\alpha^{k-2\ell-1}\sum_{i=2}^{k}{T_{\rho_0}^{-1}}\|Op_h(\twe b_\e)Op_h({\tilde{w}} \chi ){Q_i}\phi_{h}\|_{L^2(M)}^2 \notag\\
&\quad+C_{\ell,k}\alpha^{k-2\ell-1}h^{-2}\sum_{i=2}^{k}{C_0^2T}\|Op_h(\twe b_\e)POp_h({\tilde{w}}\chi ){Q_i}\phi_{h}\|_{L^2(M)}^2+o_\e(1) \label{e:rex}.
\end{align}

In particular, since $\mu$ is the defect measure associated to $\{\phi_h\}$, {arguing as in \eqref{E:bedbug} we obtain}
\begin{align*}
\limsup_{h\to 0}h^{k-1}\|Op_h(\beta_\e)&Op_h(\chi w) \phi_{h}\|_{L^2(H)}^2\leq \\
&C_{\ell,k}\alpha^{k-1} \int_{T^*M}  \twe^2b_\e^2({T_{\rho_0}^{-1}}\chi^2+{C_0^2T}|H_p({\tilde{w}}\chi )|^2)d\mu\\
&+C_{\ell,k}\alpha^{k-2\ell-1}\sum_{i=2}^{k}\int_{T^*M} \twe^2 b_\e^2({T_{\rho_0}^{-1}}\chi^2q_{i}^2+{C_0^2T}|H_p({\tilde{w}}\chi   q_{i})|^2)d\mu.
\end{align*}

Next, we observe that by~\eqref{e:platypus} and the fact that $0\leq b_\e^2\leq 1$, we have 
\[
\lim_{\e\to 0} \twe^2b_\e^2 \leq \tilde{w}^2{1_{\supp \tilde{w}}}.
\]
Sending $\e\to 0$ and using $H_p({\tilde{w}}\chi)=0$ on $\LambdaH$ (together with $\mu(T^*M)=1$ to apply the Dominated Convergence Theorem) we have
\begin{equation}
\label{e:prelimEsta}
\begin{aligned}
\lim_{\e\to 0}\limsup_{h\to 0}h^{k-1}\|Op_h(\beta_\e w)Op_h(\chi )& \phi_{h}\|_{L^2(H)}^2\leq C_{\ell,k}\alpha^{k-1}{T_{\rho_0}^{-1}} \int_{\LambdaH}  \chi^2 \tilde{w}^2 d\mu\\
&\hspace{-1cm}+C_{\ell,k}\alpha^{k-2\ell-1}\sum_{i=2}^{k}\int_{\LambdaH} \chi^2 \tilde{w}^2({T_{\rho_0}^{-1}}q_{i}^2+{C_0^2T}|H_pq_{i}|^2)d\mu.
\end{aligned}
\end{equation}

Next, assume that $\supp (\chi {1_{ \LambdaH}}) \subset  \mc{T}_{_T}(\Xi,R)$. By \cite[Lemma 3.1]{Gdefect}
\begin{equation}\label{E:distance}
\sup_{{\mathcal T_T(\Xi,R)\cap \LambdaH}} \max_i |\xi_i-a_{i}(x_1,x')| \leq 3R.
\end{equation}
Hence, since $H_p(\xi_i-a_i(x_1,x'))=0$ on $\gamma_{x'}$,
\[
\sup_{{\mathcal T_T(\Xi,R)\cap \LambdaH}}|H_pq_{i}|\leq CR^\ell.
\]
Furthermore, 
\[
\sup_{{\mathcal T_T(\Xi,R)\cap \LambdaH}}|q_{i}|\leq (1+C\delta)R^\ell+O(R^{2l})
\]
Thus, {taking $T$ small enough}, we obtain from~\eqref{e:prelimEsta} that
\begin{multline*}
\lim_{\e\to 0}\limsup_{h\to 0}h^{k-1}\|Op_h(\beta_\e w)Op_h(\chi ) \phi_{h}\|_{L^2(H)}^2\\\leq C_{\ell,k}{T_{\rho_0}^{-1}}\int_{\LambdaH}\chi^2 \tilde{w}^2( \alpha^{k-1} +\alpha^{k-2\ell-1}R^{2\ell})d\mu.
\end{multline*}
Choosing $\alpha=R$ and fixing $\ell=k$ gives~\eqref{e:optimized}.
\end{proof}
\begin{remark}\label{r:explained}
To see that the conclusion in Remark~\ref{r:uniformity} holds, observe that the  estimate in \eqref{e:rex} holds for $\tilde{H}$ as long as $\Sigma_{\tilde{H},p}$ and $\SigH$ are $o(1)$ close. Thus, it is enough that $H$ and $\tilde{H}$ are $o(1)$ close in the $C^1$ norm.
\end{remark}

\needspace{.5cm}
We now present the proof of Proposition \ref{P:sloth}.
\subsubsection{\bf Proof of Proposition \ref{P:sloth}.} \label{S: sloth proof}
Let $\chi\in C_c^\infty(T^*M)$ so that $H_p \chi\equiv 0$ on $\LambdaH$ for some $T>0$. 
Also, fix $w \in C_c^\infty(H)$.

{For all $\delta>0$, w}e can find $(x_j,r_j)$ and $(\Xi_j, R_j)$ with $j=1,\dots K(\delta)$ so that  if we set 
\[
U_j:= \{ (x,\xi):\; x\in B(x_j,r_j),\, \xi \in {B(\Xi_j(x),R_j)}\} \subset  { \SigH} \qquad \text{and}\qquad \mc{U}=\bigcup_{j=1}^K U_j,
\]
{where $B(x_j,r_j) \subset H$ and $B(\Xi_j(x),R_j) \subset \{\xi \in  N^*_xH:\; p(x,\xi)=0\} $ are balls of radius $r_j$ and $R_j$ respectively,}
then
\[
\supp(\chi {1_{\SigH}})\; \subset \; \mc{U},
\]
and
\[
\sum_{j=1}^K {\Sig}(U_j) \; \leq \;   \Sig\big(\supp (\chi {1_{ \SigH}}) \big)+{\delta}.
\]

Let $\tilde{\chi}_j$ be a partition of unity for $\mc{U}$ subordinate to $\{U_j\}$. Apply Lemma \ref{L:ant} to obtain the flow invariant extensions 
\[
\chi_j\in C_c^\infty(T^*M;[0,1])
\]
 so that 
 \begin{enumerate}
\item $H_p\chi_j\equiv 0 \text{ on }\LambdaH$,
\item $(\supp \chi_j1_{{ \LambdaH}})\subset \bigcup_{|t|<T} \varphi_t(U_j) \subset {\mathcal T_T(\Xi_j,R_j)}$, 
\item  $\{x:\; (x,\xi)\in (\supp\chi_j{1_{ T^*_HM}})\} \subset B(x_j,r_j)$,
\item $\sum_{j=1}^K\chi_j\equiv 1$ on  $ \bigcup_{|t|<T} \varphi_t( \mc{U} )$,
\item $0\leq \sum_{j=1}^K\chi_j\leq 1 \text{ on } \LambdaH$. 
\end{enumerate}

Note that, since $H_p \chi \equiv 0$ on $\LambdaH$, we have 
\[
\supp (\chi {1_{\LambdaH}}) =  \bigcup_{|t|<T}G^t(\supp \chi {1_{ \SigH}}) \subset \bigcup_{|t|<T} G^t(\mc U).
\]
Therefore, 
$$
\supp \Big(1-\sum_{j=1}^K \chi_j\Big)\cap \supp( \chi{1_{ \LambdaH}})=\emptyset.
$$
By Lemma~\ref{L:optimized}, we conclude
$$
\lim_{\e\to 0}\limsup_{h\to 0} h^{\frac{k-1}{2}}\left|\int_H Op_h(\beta_\e w)\Big[Op_h\Big(1-\sum_{j=1}^K\chi_j\Big)Op_h(\chi)\phi_h\Big]d\sigma_H\right|=0.
$$

We then have 
\begin{align*}
&\lim_{\e\to 0}\limsup_{h\to 0} h^{\frac{k-1}{2}}\left|\int_H Op_h(\beta_\e w)[Op_h(\chi)\phi_h]d\sigma_H\right|=\\
&\hspace{3cm}=\lim_{\e\to 0}\limsup_{h\to 0} h^{\frac{k-1}{2}}\left|\int_H Op_h(\beta_\e w)\Big[Op_h\Big(\sum_{j=1}^K\chi_j\Big)Op_h(\chi)\phi_h\Big]d\sigma_H\right|.
\end{align*}
Now, to recover the spatial localization we introduce $\psi_j\in C_c^\infty(H)$ with $\supp \psi_j\subset B(x_j,2r_j)$ and
\[
\psi_j(x')\chi_j(0,x',\xi)=\chi_j(0,x',\xi),\qquad (x', \xi)\in T_H^*M.
\]
Then,
\[
\|Op_h(\chi_j)\phi_h\|_{L^2(H)}=\|\psi_j Op_h(\chi_j)\phi_h\|_{L^2(H)}+O(h^{\frac{2-k}{2}}).
\]
{In fact, on $\R^d$ with the standard quantization, we have $[(1-\psi_j)Op_h(\chi_j)\phi_h]|_{H}=0$. Hence, the above estimate follows from the fact that quantizations differ by $O_{L^2\to L^2}(h)$ together with the standard restriction estimate for compactly microlocalized functions.}

In what follows we bound $\|Op_h(\beta_\e)[Op_h(\chi_j\chi)\phi_h]\|_{L^2(H)}$ using Lemma  \ref{L:optimized} applied to $\chi_j \chi$. This can be done since $H_p(\chi \chi_j) \equiv 0$ on $\LambdaH$. 
Lemma \ref{L:optimized} yields that there exists {$C_{k}>0$ depending only on $k$ and $\rho_j\in (B(x_j,3r_j)\times B(\Xi(x_j),3R_j))\cap \SigH$} so that, {for any  $\tilde{w}\in C^\infty_c(T^*M)$  extension of $w$ with $H_p \tilde w\equiv 0$ on $\LambdaH$, and {$T_{\rho_j}:=T|\partial_{\xi_1}p(\rho_j)|$,}
\begin{align*}
&\lim_{\e\to 0}\limsup_{h\to 0} h^{\frac{k-1}{2}}\left|\int_H Op_h(\beta_\e w)[Op_h(\chi)\phi_h]d\sigma_H\right|\\
&\hspace{2cm}\leq \lim_{\e\to 0}\limsup_{h\to 0} h^{\frac{k-1}{2}}\sum_{j=1}^K\|1_{\supp \psi_j}\|_{L^2(H)}\|Op_h(\beta_\e w)[Op_h(\chi_j\chi )\phi_h]\|_{L^2(H)}\\
&\hspace{2cm}\leq C_k\sum_{j=1}^K\|1_{\supp \psi_j}\|_{L^2(H)}\left({T_{\rho_j}^{-1}}R_j^{k-1}\int_{\LambdaH}\chi_j^2 \chi^2 \tilde{w}^2 d\mu\right)^{1/2}\\
&\hspace{2cm}\leq C_k\sum_{j=1}^K r_j^{\frac{n-k}{2}}R_j^{\frac{k-1}{2}}\left({T_{\rho_j}^{-1}}\int_{\LambdaH}\chi_j^2\chi^2 \tilde{w}^2 d\mu\right)^{1/2}\\
&\hspace{2cm}\leq C_k \left(\sum_{j=1}^K r_j^{n-k}R_j^{k-1}\right)^{1/2}\left({T_{\rho_j}^{-1}}\int_{\LambdaH}\sum_{j=1}^K \chi_j^2\chi^2 \tilde{w}^2 d\mu\right)^{1/2}\\
&\hspace{2cm}\leq C_k  c_{n,k} ^{1/2}  \Big[\Sig\big(\supp \chi{1_{ \SigH}}\big)+{\delta}\Big] ^{1/2}\left(\int_{\SigH}\chi^2 w^2 {|H_pr_H|^{-1}} d\muH+{\delta}\right)^{1/2}.
\end{align*}
We have used  that there exists ${c_{n,k}=c(n,k)>0}$ so that {for $r_j$ and $R_j$ small enough} 
$$
\sum_{j=1}^K r_j^{n-k}R_j^{k-1} \; \leq\;  c_{n,k}\sum_{j=1}^K \Sig(U_j) \leq  c_k\, \Big[\Sig\big(\supp \chi {1_{\SigH }}\big)+{\delta}\Big],$$
{ that by continuity of $|H_pr_H|^{-1}$ on $\SigH$, as $r_j,R_j\to 0$, 
$$
\sum_j\chi^2\chi_j^21_{U_j}\sup_{U_j}|H_pr_H|^{-1}\to \chi^2|H_pr_H|^{-1},
$$
and the dominated convergence theorem.
}
 Since ${\delta}>0$ is arbitrary, this completes the proof of the proposition.
\ \\
\qed
\ \\

\subsubsection{\bf Proof of Lemma \ref{l:squirrel}.}\label{S: squirrel proof}
First, suppose {$\rho_0\in T^*M$ is so that $p(\rho_0)\neq 0$}. Then, there exists a neighborhood  $U\subset T^*\R^n$  of ${\rho_0}$ with $U\subset \{p\neq 0\}$. One can then carry an  elliptic parametrix construction so that  
\begin{equation}\label{e:elliptic}
 Op_h(q\, \kappa \, \chi)\phi_h=Op_h(\tilde{e}){Op_h(\tilde{\chi})Op_h(p)\phi_h},
\end{equation}
for all $\chi$ supported in $U$ and some suitable $\tilde e$. Therefore, 
$$
{ \|Op_h(q\, \kappa \, \chi){\phi_h}(0,x')\|_{{L^2_{ \tilde{x}}}}\leq C\|Op_h(\tilde{\chi})P\phi_h\|_{L^2_x}},
$$
{as claimed.}
We may assume from now on that  \[
{\rho_0}\in \{\partial_{\xi_1}p\neq0\}\cap \{p= 0\}.
\] 
By the implicit function theorem, for $ {\tilde{\chi}}$ supported sufficiently close to ${\rho_0}$, {and $\supp \chi\subset\{ {\tilde{\chi}}\equiv 1\}$}
\[
p(x,\xi) {\tilde{\chi}(x,\xi)}=e(x,\xi)(\xi_1-a(x,{\tilde \xi}))
\]
with $e(x,\xi)$ elliptic on $\supp \chi$ {and $\xi=(\xi_1,\tilde{\xi})$}. In particular,
$$
Op_h(p){Op_h(\chi)}=Op_h(e)(hD_{x_1}-Op_h(a)))Op_h(\chi)+hOp_h(R){Op_h(\chi)}.
$$
Therefore, 
$$
(hD_{x_1}-Op_h(a)) w=f,
$$
where we have set{
\begin{gather*}
w:=Op_h(\chi )Op_h(q)\phi_h,\\
f:=[Op_h(e)^{-1}Op_h(p)Op_h(\chi )Op_h(q)+h Op_h(R_1)Op_h(\chi)Op_h(q)]\phi_h+O(h^\infty)
\end{gather*}
and $Op_h(e)^{-1}$ denotes a microlocal parametrix for $Op_h(e)$ near $\supp \chi$.}
Defining
\[
A(t,s,\tilde{x},hD_{\tilde{x}}):=-\int_{s}^{t}a(x_1,\tilde{x},hD_{\tilde{x}})dx_1,
\] 
 we obtain that for all $s,t \in \R$
$$
w(s,\tilde{x})=e^{-\frac{i}{h}A(t,s,\tilde{x},hD_{\tilde{x}})}w(t,\tilde{x})\\
-\frac{i}{h}\int_{s}^{t}e^{-\frac{i}{h}A(x_1,s,\tilde{x},hD_{\tilde{x}})}f(x_1,\tilde{x})dx_1.
$$
Let $\delta>0$ be so that 
\begin{equation}
\label{E:delta}
\delta\leq  \frac{T}{3}|\partial_{\xi_1}p(\rho_0)| \leq \frac{T}{2}\inf \Big\{  |\partial_{\xi_1}p(x,\xi)|:\; (x,\xi)\in \supp \chi \Big\}
\end{equation}
and $\Phi\in C_c^\infty(\re;[0,{2\delta^{-1}}])$ with $\supp \Phi \subset [0,\delta]$ and $\int_\R \Phi=1$. Then, integrating in $t$,
\[
w(s,\tilde{x})=\int_{\R}\!\Phi (t)e^{-\frac{i}{h}A(t,s,\tilde{x},hD_{\tilde{x}})}w(t,\tilde{x})dt-\frac{i}{ h}\int_\R\! \Phi(t) \int_{s}^{t}e^{-\frac{i}{h}A(x_1,s,\tilde{x},hD_{\tilde{x}})}f(x_1,\tilde{x})dx_1dt.
\]
Next, applying propagation of singularities, we claim that
\begin{equation}
\label{e:est0}
\begin{aligned} 
Op_h(\kappa) w(s,\tilde{x})&=\int_\R\!\Phi(t)Op_h(\kappa) e^{-\frac{i}{h}A(t,s,\tilde{x},hD_{\tilde{x}})}Op_h(b)w(t,\tilde{x})dt\\
& -\frac{i}{h}  \int_\R \! \Phi(t)\int_{s}^{t}Op_h(\kappa)e^{-\frac{i}{h}A(x_1,t,\tilde{x},hD_{\tilde{x}})}Op_h(b)f(x_1,\tilde{x})dx_1dt\\
&+{R_h(s, \tilde x)}+O(h^\infty)\|\phi_h\|_{L^2},
\end{aligned}
\end{equation}
{with $\|R_h(s, \tilde x)\|_{L^\infty_{s}L^2_{\tilde{x}}}=O(h^{-1}\|Op_h(\tilde{\chi})P\phi_h\|_{L^2_x})$.}
Indeed, \eqref{e:est0} follows once we show that for any $v\in S^0(T^*M)$ {supported on $\tilde{\chi}\equiv 1$} and $x_1\in[0,\delta]$ 
\begin{multline}\label{e:est1}
\|Op_h(\kappa)e^{-\frac{i}{h}A(x_1,t,\tilde{x},hD_{\tilde{x}})}(1-Op_h(b))Op_h(v)\phi_h\|_{L^2_x}\leq \\{C}\|Op_h(\tilde{\chi})P\phi_h\|_{L^2_x}+O(h^\infty)\|\phi_h\|_{L_x^2}.
\end{multline}
Let $\chi_\e \in C_c^\infty(\R;[0,1])$ be as in \eqref{E:chi}. By the same construction carried in \eqref{e:elliptic} (which gives that $\phi_h$ is microlocalized on $\{p=0\}$) we conclude
\begin{multline}
\label{e:elliptPart}
\|Op_h(\kappa)e^{-\frac{i}{h}A(x_1,t,\tilde{x},hD_{\tilde{x}})}(1-Op_h(b))Op_h(v)(1-Op_h(\chi_\e(p))\phi_h\|_{L_x^2}\leq\\ 
{C_{\e}}\|Op_h(\tilde{\chi})P\phi_h\|_{L_x^2}+O(h^\infty)\|\phi_h\|_{L^2_x}.
\end{multline}
Therefore, to prove \eqref{e:est1} we need to estimate
\begin{equation}
\nonumber
\|Op_h(\kappa)e^{-\frac{i}{h}A(x_1,t,\tilde{x},hD_{\tilde{x}})}(1-Op_h(b))Op_h(v)Op_h(\chi_\e(p)) \phi_h\|_{L_x^2}.
\end{equation}
Let $\tilde{\varphi}_t$ denote the Hamiltonian flow of $\tilde{p}(x,\xi)=\xi_1-a(x,{\tilde \xi})$. Then, for $(x,\xi)\in \{(x,\xi): \;|p(x,\xi)|\leq C\e^2\}$ and $|t|\leq 1$, we have $d(\varphi_t(x,\xi),\tilde{\varphi}_t(x,\xi))\leq C\e^2$. 
By~\eqref{e:platypus}, $b$ is identically 1 in a neighborhood of  
\[
\bigcup_{|t|\leq T} \varphi_t(\{\supp \kappa\}\cap \{p=0\})
\] 
and thus for $\e>0$ small enough on 
\[
\bigcup_{|t|\leq2T}\tilde{\varphi}_t(\{ \supp \kappa\}\cap\{ |p|\leq C\e^2\}).
\]
In particular, since we assume that $\supp \Phi\subset [0,\delta]$ and $\delta$ satisfies \eqref{E:delta},  we have
\begin{equation}
\label{e:propPart}
\|\Phi(t)Op_h(\kappa)e^{-\frac{i}{h}A(t,s,\tilde{x},hD_{\tilde{x}})}(1-Op_h(b))Op_h(a)Op_h(\chi_\e(p))\phi_h\|_{L^2_x}=O_{\e}(h^\infty)\|\phi_h\|_{L^2_x}.
\end{equation}
Together~\eqref{e:elliptPart} and~\eqref{e:propPart} give~\eqref{e:est1}.  {In particular, we obtain}~\eqref{e:est0} which, since {
$$
\Phi(t) \leq 2\delta^{-1},\qquad\text{ and hence } \qquad \|\Phi\|_{L^2}\leq C,
$$} implies
\begin{multline*}
\|Op_h(\kappa) w(0,\cdot)\|_{L^2_{\tilde{x}}}\leq
 2{\delta^{-1/2}}\|Op_h(b)w\|_{L^2_x}+C_0{\delta^{1/2}} h^{-1}\|Op_h(b)f\|_{L^2_x}\\
 +Ch^{-1}\|Op_h(\tilde{\chi})P\phi_h\|_{L^2_x}+O(h^\infty)\|\phi_h\|_{L_x^2},
\end{multline*}
Now, 
$$
Op_h(q)Op_h(\kappa)Op_h(\chi)=Op_h(\kappa)Op_h(\chi)Op_h(q)+[Op_h(q),Op_h(\kappa)Op_h( \chi)].
$$
Therefore, since
$$
\|[Op_h(q),Op_h(\kappa)Op_h( \chi)]\phi_h(0, \cdot )\|_{L_{\tilde{x}}^2} \leq Ch^{\frac{1}{2}}\|Op_h(\tilde{\chi})\phi_h\|_{L^2_x}+O(h^\infty)\|\phi_h\|_{L^2_x},
$$
we have the following $L^2$ bound along the section $x_1=0$ 
\begin{multline} 
\label{local upshot 1}
\|Op_h(\kappa)Op_h(\chi)Op_h(q)\phi_h(0, \cdot) \|_{L^2_{\tilde{x}}}\leq 
 2{\delta^{-1/2}} \|Op_h(b)w\|_{L^2_x}+C_0 {\delta^{1/2}} h^{-1}\|Op_h(b)f\|_{L^2_x}\\
 {+Ch^{-1}\|Op_h(\tilde{\chi})P\phi_h\|_{L^2_x}
 +Ch^{\frac{1}{2}}\|Op_h(\tilde{\chi})\phi_h\|_{L^2_x}+O(h^\infty)\|\phi_h\|_{L^2_x}}
 \end{multline}
 finishing the proof.
 %

\qed

\section{Proof of Theorem  \ref{thm:normal}}\label{S:normal}
When the codimension of $H$ is equal to $1$ and $\SigH$ is compact we can include an estimate on the normal derivate in all of our results. In particular, for $\nu$ a unit normal to $H$, we may {\bf replace  all instances} of $\int_A \phi_hd\sigma_H$ with
\[
\Big|\int _A\phi_hd\sigma_H\Big|+\Big|\int_AhD_\nu \phi_hd\sigma_H\Big|.
\]
To see this, observe that if $\phi_h$ is a quasimode for $P$ {and $\{\phi_h\}$ is compactly microlocalized}, then 
\[
hD_\nu P\phi_h=o(h).
\] 
In particular, 
\begin{equation}\label{E:commutator}
PhD_\nu \phi_h+[hD_\nu, P]\phi_h=o(h).
\end{equation}
Let $\chi \in S^0(T^*M)$ have $\chi \equiv 1$ in a neighborhood of $N^*\!H$ and 
\[
\supp \chi \subset \Big\{ (x, \xi)\in T^*M:\, |\langle \nu(x),\xi\rangle|>\frac{|\xi|}{2}\Big\}.
\]
Then, there exists $E\in \Psi^\infty(M)$ so that 
\[
Op_h(\chi)[hD_\nu, P]=hEhD_\nu 
\]
and in particular, applying $Op(\chi)$ to \eqref{E:commutator} we find
\[
(Op_h(\chi)P+hE)hD_\nu \phi_h=o(h).
\]
Now, $\sigma(Op_h(\chi)P+hE)=\chi \,p$. Therefore, since $\chi \equiv 1$ in a neighborhood of $N^*\!H$ and $H$ is conormally transverse for $p$, $H$ is conormally transverse for $\chi(x,\xi)p(x,\xi)$. Thus, Theorem~\ref{t:local2} applies and gives 
\[
\limsup_{h\to 0^+}\left|\int_A w hD_\nu \phi_hd\sigma_H\right|\leq {C_{n,k}}\int_{\pi_H^{-1}(A)} |w| \sqrt{\tilde{f}|H_pr_H|^{-1}} d\Sig,
\]
where 
\[
\tilde{\mu}_{H,\chi p}=\tilde{f} d\Sig+\tilde{\lambda}_H
\]
with $\tilde{\lambda}_H\perp \Sig$ and $\tilde{\mu}$ is the defect measure for $hD_\nu \phi_h$. It is straightforward to see that 
\[
\tilde{\mu}=|\langle \nu(x),\xi\rangle|^2\, \mu,
\]
and hence (for $t_0>0$ chosen small enough)
\[
\tilde{\mu}_{H,\chi p}=|\langle \nu(x),\xi\rangle|^2\mu_{_{\!H,p}}=|\langle \nu(x),\xi\rangle|^2(fd\Sig+\lambda_{_{\! H}}).
\]
In particular, 
\begin{align*}
\limsup_{h\to 0^+}\left|\int_A w hD_\nu \phi_hd\sigma_H\right|&\leq {C_{n,k}}\int_{\pi_H^{-1}(A)} |w| \sqrt{f|H_pr_H|^{-1}}|\langle \nu(x),\xi\rangle| d\Sig \\
&\leq \tilde{C}\int_{\pi_H^{-1}(A)} |w| \sqrt{f|H_pr_H|^{-1}} d\Sig,
\end{align*}
since $\SigH$ is compact {and $f$ is supported on $\SigH$}.
\begin{remark}
{Note that the constant $\tilde{C}$ now depends on $\sup_{\SigH}|\langle \nu(x),\xi\rangle|$. }
\end{remark}

This proves that the analog of  Theorem~\ref{t:local2} holds for $hD_\nu \phi_h$.  One can then obtain an analog of Theorem~\ref{t:local} for $hD_\nu \phi_h$, which in turn implies Theorem~\ref{thm:normal}.

\section{Proof of Theorem  \ref{thm:recur}}\label{S:recur}

{We prove Theorem  \ref{thm:recur} by contradiction.} {Suppose that there exists a sequence $\{\phi_{h_m}\}$ and $c>0$ such that 
\begin{equation}
\label{e:contra1}
\Big|\int_A\phi_{h_m}d\sigma_H\Big|\geq ch_m^{\frac{1-k}{2}}.
\end{equation}
Then, we may extract a subsequence (still writing it as $\phi_{h_m}$) with defect measure $\mu$.} Let $\muH$ be the induced measure on $\SNH$ and $\lambda_{_{\! H}}$ be the measure on $\SNH$ with $\lambda_{_{\! H}} \perp \sig$ and so that 
\[
\muH= f\, \sig + \lambda_{_{\! H}},
\]
for $f \in L^1(\SNH , \sig)$.
 {Then,
\begin{align}
\int_{\pi_H^{-1}(A)} \sqrt{f} \,d\sig &= \int_{ \mc{R}_H \cap \pi_H^{-1}(A)} \sqrt{f} \,d\sig + \int_{ \mc{R}_H^c \cap \pi_H^{-1}(A)} \sqrt{f} \,d\sig \notag\\
&= \int_{ \mc{R}_H^c \cap \pi_H^{-1}(A)} \sqrt{f} \,d\sig, \label{this}
\end{align}
where the last equality follows from the fact that  $\sig(\mc{R}_H \cap  \pi_H^{-1}(A))=0$.
Also,
since $\lambda_{_{\! H}}\perp \sig$, there exist $V,W\subset \SNH$ so that $\lambda_{_{\! H}}(W)=\sig(V)=0$ and $\SNH=V\cup W$. 
Next, we use that Lemma \ref{P:suppRecur} below gives $\muH(\mc{R}_H^c)=0$. It follows that
\begin{equation}\label{that}
 \int_{ \mc{R}_H^c \cap \pi_H^{-1}(A)} \sqrt{f} \,d\sig
 \leq   \left(\int_{ \mc{R}_H^c \cap \pi_H^{-1}(A)} f  \,d\sig\right)^{\frac{1}{2}}
=\muH(\mc{R}_H^c\cap \pi_H^{-1}(A)\cap W)^{\frac{1}{2}}=0.
\end{equation}
}
%
Combining \eqref{this} and \eqref{that} gives $\int_{\pi_H^{-1}(A)} \sqrt{f} \,d\sig=0$, and so Theorem \ref{t:local}  gives a contradiction to~\eqref{e:contra1}.

\qed

\begin{lemma}
\label{P:suppRecur}
Let $H\subset M$ and suppose that $\{\phi_h\}$ is a sequence of eigenfunctions with defect measure $\mu$. 
Then, 
\[
\muH(\mc{R}_H)=\muH(\SNH).
\]
\end{lemma}
 
\begin{proof}
Let $B\subset \SNH$ be an open set and for $\delta>0$ define 
\[
B_{2\delta}:=\bigcup_{-2\delta<t<2\delta} G^t(B).
\]
Observe that the triple $(\SM, \mu,G^t)$ forms a measure preserving dynamical system.
 The Poincar\'e Recurrence  Theorem \cite[Lemma 4.2.1, 4.2.2]{BrinStuck} implies that for $\mu$-a.e. $\rho \in B_{2\delta}$ there exist $t^{\pm}_n\to \pm\infty$ so that $G^{t^{\pm}_n}(\rho)\in B_{2\delta}$. By the definition of $B_{2\delta}$, there exists $s^{\pm}_n$ with $|s^{\pm}_n-t^{\pm}_n|<2\delta$ such that $G^{s^{\pm}_n}(\rho)\in B$. In particular,  for $\mu$-a.e. $\rho\in B_{2\delta}$,
\begin{equation}
\label{e:recurs}
\begin{gathered}
\bigcap_{T>0}\overline{\bigcup_{t\geq T}G^t(\rho)\cap B}\neq \emptyset,\qquad \text{and}\qquad 
\bigcap_{T>0}\overline{\bigcup_{t\geq T}G^{-t}(\rho)\cap B}\neq \emptyset.\\
\end{gathered}
\end{equation}
 We have used that the sets $\overline{\cup_{t\geq T}G^{\pm t}(\rho)\cap B}$  are non-empty, compact, and nested as $T$ grows.

 We next show that~\eqref{e:recurs} holds for $\muH$-a.e. point in $B$.  To do so, suppose the opposite. Then, there exists $A\subset B$ with $\muH(A)>0$ so that for each $\rho\in A$, there exists $T>0$ with
\begin{equation}
\label{e:norecur}
\bigcup_{t\geq T}G^{t}(\rho)\cap B =\emptyset \qquad \text{or}\qquad   \bigcup_{t\geq T}G^{-t}(\rho)\cap B=\emptyset.
\end{equation}
We relate $\mu$ and $\muH$ using  \cite[Lemma 6]{CGT} which gives
\[\mu|_{B_{2\delta}}=\muH dt.\]
Then, if we let
\[
A_\delta:=\bigcup_{-\delta<t<\delta} G^t(A),
\]
we have
\[
\mu(A_\delta)=2\delta \cdot\muH(A)>0.
\]
Then $A_\delta\subset B_{2\delta}$, and for all $\rho \in A_\delta$ there exists $T>0$ so that \eqref{e:norecur} holds.
Since this implies that \eqref{e:recurs} does not hold for a subset of $B_{2\delta}$ of positive $\mu$ measure, we have arrived at a contradiction.
Thus~\eqref{e:recurs} holds for $\muH$ a.e. point in ${B}$.  

To finish the argument, let $\{B_k\}$ be a countable basis for the topology on $\SNH$. Then for each $k$ there is a subset $\tilde{B_k} \subset B_k$ of full $\muH$ measure so that for \emph{every} $\rho \in\tilde{B}_k$ relation \eqref{e:recurs} holds with $B=B_k$.

{Let $X_k:=\tilde{B}_k\cup(\SNH\setminus B_k)$}. {Next, note that $\cap_k X_k\subset \mc{R}_H$. Indeed, if $\rho \in \cap_k X_k$ and $\mathcal U \subset \SNH$ is an open neighborhood of $\rho$, then there exists $\ell$ so that $\rho \in B_\ell \subset \mathcal U$. In particular, since $\rho \in X_\ell$, we know that $\rho \in \tilde B_\ell$ and so  $\bigcap_{T>0}\overline{\bigcup_{t\geq T}G^t(\rho)\cap B_\ell}\neq \emptyset.$ We conclude that $\rho$ returns infinitely oftern to $\mathcal U$.}

 Noting that $X_k=\tilde{B}_k\cup (\SNH\setminus {B_k})$ has full $\muH$ measure, we conclude that $\cap_k X_k\subset \mc{R}_H$ has full measure and thus $\muH(\mc{R}_H\cap \SNH)=\muH(\SNH)$ as claimed.
\end{proof}

\section{Recurrence: Proof of Theorem \ref{T:applications} } \label{S:applications}

This section is dedicated to the proof of Theorem \ref{T:applications}. In Section~\ref{S:parts1} we prove the theorem for assumptions \ref{a1} and \ref{a2} by showing that  $\sig({\mc L}_H)=0$.  
In Section~\ref{S:parts2} we present a tool for proving that  $\sig({\mc R}_H \cap A)=0$  for $A \subset \SNH$. 
In particular, we prove that it suffices to show that $t \mapsto \vol(G^t(A))$ is integrable either for positive times or for negative ones. 
 In Section~\ref{S:parts4} we show that for manifolds with  Anosov flow we have $\sig({\mc R}_H)=\sig({\mc R}_H \cap \mc{A}_H)$, where $\mc{A}_H$ is the set of points in $\SNH$  at which the tangent space to $\SNH$ splits into a direct sum of stable and unbounded directions. {A similar statement is proved for $(M,g)$ with no focal points, but with $\mc{N}_H$ instead of $\mc{A}_H$.}
 In Section~\ref{S:parts3} we  prove Theorem \ref{T:applications} for assumptions \ref{a3}, \ref{a4}, \ref{a5} and \ref{a6}, by taking advantage of the fact  when $(M,g)$ has Anosov flow we have some control on the structure of $\mc{A}_H$ and, in some cases, on the integrability of $t \mapsto \vol(G^t(\mc{A}_H))$.

\subsection{Proof of  parts {\bf \ref{a1}} and {\bf \ref{a2}}  } \label{S:parts1}

In this section we prove that $\sig({\mc R}_H)=0$ for $(M,g)$ and $H$ satisfying the assumptions in parts \ref{a1} and \ref{a2} in Theorem \ref{T:applications}.\ \\ \ \\
\noindent{\bf Proof of  part {\bf \ref{a1}}.} For this part we assume that $(M,g)$ has no conjugate points and $H$ has codimension $k>\frac{n{+1}}{2}$. 
{The strategy of the proof is to show that the set $\{\rho \in \SNH:\; \exists t>0\;\text{s.t.}\; G^t(\rho)\in {\SNH}\}$ has dimension strictly smaller than $n-1=\dim \SNH$, and hence has measure zero. We prove this using the implicit function theorem together with the fact that, since $(M,g)$ has no conjugate points, we can control the rank of the exponential map.}

Note that, since $(M,g)$ has no conjugate points, for each point $x\in M$ the exponential map $\exp_x:T_xM \to M$ has no critical points. In particular, if we define the map 
\[
\psi^x: \R \times \snx \to M, \qquad \qquad \psi^x(t ,\xi)= \exp_x(t\xi),
\] 
we have for all $(t,\xi) \in  \R \times \snx $
 \[
 \text{rank}\, (d\psi^x)_{(t,\xi)} = n-\dim H.
 \]
 This implies that if we define
\[
\psi: \R \times \SNH \to M, \qquad \qquad  \psi (t,\rho)=\pi G^t(\rho),
\]
then its differential
\[
(d\psi)_{(t,\rho)}:T_{(t,\rho)}(\R\times \SNH)\to T_{\pi G^t(\rho)}M
\]
has  
 \[
 \text{rank} (d\psi)_{(t,\rho)}\geq n-\dim H=k,
 \]
 for all $(t,\rho) \in  \R \times \SNH$. { Note that $\psi^{-1}(H)=\{(t, \rho) \in \R \times \SNH:\; G^t(\rho)\in S_H^*M\}$.}
  
 Let 
 \begin{equation}
 \label{e:defFunc}
 \begin{gathered}
 f_i\in C^\infty(M;\R),\qquad F=(f_1, \dots, f_k):M \to \R^k,\\
F^{-1}(0)=H,\qquad 
 \{df_i\}_{i=1}^k\text{ linearly independent on }H.
 \end{gathered}
 \end{equation}
  The composition $F\circ \psi : \R \times \SNH \to  \R^k$ satisfies $(F \circ \psi)^{-1}(0)=\psi^{-1}(H)$. 
 Note that since $ \text{rank} (d\psi)_{(t,\rho)} \geq k$, we have
 \[
  \text{rank} (d (F \circ \psi)_{(t,\rho)}) \geq \text{rank} (dF)_{\psi(t,\rho)}+  \text{rank} (d\psi)_{(t,\rho)}   - \dim M \geq  2k -n
 \]
for $(t, \rho) \in (F \circ \psi)^{-1}(0)$. 
Since by assumption $k>\frac{n{+1}}{2}$, we have 
 \[
  \text{rank} (d (F \circ \psi)_{(t,\rho)}) \geq 2. 
  \]
  {Moreover, since the geodesic flow is transverse to $H$ along $N^*H$, $d(F \circ \psi)_{(t,\rho)}\partial_t\neq 0$ whenever $G^t(\rho)\in \SNH$.}
  {Indeed, suppose that  $G^t(\rho)\in \SNH$ and  $d(F \circ \psi)_{(t,\rho)}\partial_t= 0$. Then, if we write  $(x_t,\xi_t)=  G^t(\rho)$, 
  we have that $(df_j)_{x_t}(\xi_t)=0$ for all $j=1, \dots, k$, and this contradicts the assumption that  $\{(df_j)_x:\; j=1, \dots, k\}$ are linearly independent {and span $N^*\!H$} for $x\in H$.}

Applying the Implicit Function Theorem, we see that given $(t_0,\rho_0) \in \psi^{-1}(H)$ with {$G^{t_0}(\rho_0)\in \SNH$}, there exists a neighborhood $U$ of $(t_0,\rho_0)$, an open neighborhood $V\subset \R^{\ell}$ of $0$ for some $\ell \leq n-2$, and smooth functions $s:\SNH\to \R$, $f:V\to \SNH$ with $s(\rho_0)=t_0$, $f(0)=\rho_0$,   so that  
\[
 U \cap \psi^{-1}(H)=\{ \big(s(f(q)),f(q)\big): \; q\in V\}.
\]
In particular, since $\dim V< n-1=\dim(\SNH)$, 
\[
\sig\Big( \rho \in \SNH:\;  \text{ there exists }t\text{ such that }(t,\rho)\in U \; \text{and}\;G^t(\rho)\in S^*_HM\Big)=0.
\]
In particular, by compactness, for any $j>0$, 
\[
\sig\Big(\rho \in \SNH:\; \text{ there exists }t\in[0,j]\text{ such that }G^t(\rho)\in \SNH \Big)=0.
\]
Taking the union over $j>0$ we find 
\[\sig(\mc{L}_H)=0.\] 
 In particular, since ${\mc L}_H \supset {\mc R}_H$, this implies that $\sig({\mc R}_H)=0$.
\qed

\ \\

\noindent{\bf Proof of  part {\bf \ref{a2}}.} 
{Now, suppose that $(M,g)$ has no conjugate points and $K\subset M$ is a geodesic sphere. Then there exists $p\in M$ and $t\in \R$ so that $K=H_t$ for $H=\{p\}$. 
Applying the result in Part \ref{a1}  gives that $\sig(\mc{R}_{H})=0$. In particular, by Lemma~\ref{l:flowRecur} below we conclude
 $\sigma_{_{\!\!S\!N^*\!H_t}}(\mc{R}_{H_t})=0$ as claimed.}\\
\qed
\begin{lemma}
\label{l:flowRecur}
Suppose that $H\subset M$ is a submanifold and for $t\in \R$ define $H_t:=\pi G^t(\SNH)$. Then, for any $t\in \R$ so that $H_t$ is a smooth submanifold of $M$ {having codimension $1$}
\begin{align*}
{\Sig(\mc{R}_H)=0} &&\text{ if and only if }&& {\sigma_{\SNH_t}(\mc{R}_{H_t})=0.}
\end{align*}
\end{lemma}

\begin{proof}
First, observe that if $H\subset M$ is a submanifold, then for $t\in \R$ and $H_t:=\pi G^t(\SNH)$, we have 
\[S\!N^*\!H_t=G^t(\SNH)\sqcup G^{-t}(\SNH)\]
whenever $H_t$ is a smooth submanifold of $M$. {To see this, observe that since $H_t$ has codimension 1, for each $x\in H_t$, there are exactly two elements in $\SNH$ and hence these elements are given by 
$$
G^t(x,\xi)\qquad\text{ and }\qquad G^{-t}(x,-\xi)
$$
for some $(x,\xi)\in \SNH$. 
}
Note that {$\mc{R}_{H_t}=G^t(\mc{R}_H)\cup G^{-t}(\mc{R}_H)$}. Therefore, since $G^{\pm t}:\SNH\to S\!N^*\!H_t$ is a diffeomorphism onto its image, {$\sig(\mc{R}_{H})=0$ if and only if $\sigma_{\SNH_t}(\mc{R}_{H_t})=0$}.
\end{proof}

\subsection{A tool for proving that $\sig({\mc R}_H)=0$.} \label{S:parts2}
\ \\

Given $X \subset S^*M$ submanifold, we write $\vol(X)$ for the volume induced by the Sasaki metric on $X$.
 This section is dedicated to showing that $\sig({\mc R}_H\cap A)=0$ whenever the map $t \mapsto \vol(G^t(A))$  is integrable either on $(0, \infty)$ or on $(-\infty, 0)$. 
 We will later use that the integrability of this function can always be established if $(M,g)$ has Anosov flow and $A$ is a set of points in $\SNH$ at which the tangent space is either stable or unstable.


We start with a lemma where we prove that for any $\rho \in \SNH$ the tangent space $T_{\rho}(\SNH)$ has no component in the direction of $ \re H_p$.

\begin{proposition}\label{p:aardvark}
Let $(M,g)$ be a Riemannian manifold, and let $H \subset M$ be a {submanifold}.
For all $\rho \in \SNH$ let $\pi_{H_p}:T_{\rho}(S^*M) \to \re H_p$ be the orthogonal projection map, where $H_p$ is the Hamiltonian flow associated to $p(x,\xi)=|\xi|_{g(x)}$.
Then,
\[\pi_{H_p}(T_{\rho}(\SNH))=\{0\}.\]
\end{proposition}

\begin{proof}
Let $(x', {x''})$ be Fermi coordinates near $H$ where we identify $H$ with $\{(x',x''):\; x''=0\}$. {Writing $(\xi', \xi'')$ for the associated cotangent coordinates, }
\[
N^*\!H=\Big\{(x',0,0,\xi''):\; x' \in H,\, \xi''\in \R^{k}\Big\}.
\]
This implies that,{ if $\rho=(x', 0,0,\xi'') \in N^*H$, then }
\[
T_\rho(N^*\!H)= \{ \langle v,\partial_{x'}\rangle+\langle w,\partial_{\xi''}\rangle:  v\in \re^{n-k},\,w\in \re^k\\,\}
\]
while
{ 
\[
(H_{p})_{(x,\xi)}=\langle \xi'', \partial_{x''}\rangle \qquad \quad (x, \xi) \in \SNH.
\]}
{Now, $\partial_{x''}$ is orthogonal to $\partial_{x'}$. Thus, since $\partial_{\xi''}$ is vertical and $H_p$ is horizontal {$\R H_p$} is orthogonal to $T\SNH$.}
\end{proof}

\begin{lemma}
\label{P:integral}
Let $A\subset \SNH$.
\begin{gather}
\text{If}\;\; \int_0^\infty\vol(G^t(A))dt<\infty, \quad \text{then} \quad  \sig(\mc{L}_H^{-\infty} \cap A)=0. \label{e:spider1}\\
\text{If}\;\; \int^0_{-\infty}\vol(G^t(A))dt<\infty, \quad \text{then} \quad \sig(\mc{L}_H^{+\infty} \cap A)=0. \label{e:spider2}
\end{gather}
In particular, either assumption implies that $ \sig(\mc{R}_H \cap A)=0$.
\end{lemma}

\begin{proof}
 Suppose~\eqref{e:spider1} holds.
From now on, given $\rho \in \SNH$ and $t \in \R$, we adopt the notation
\begin{equation}\label{E:J}
J_t(\rho):=dG^t|_{T_{\rho}(\SNH)}: T_{\rho}(\SNH) \to dG^t(T_{\rho}(\SNH)). 
\end{equation}
Note that 
\[\int_A |\det  J_t(\rho) |\; d\sig(\rho)=\vol(G^t(A)).\]

We claim that there exist constants $C,\delta>0$ so that for any Borel set  $A\subset \SNH$ and $T\in \R$, 
\begin{equation}\label{E:volume}
\sig\left(\bigcup_{t=T}^{T+\delta}G^t(A)\cap \SNH\right)\leq C\int_A|\det J_T(\rho)|\; d\sig (\rho).
\end{equation}

We postpone the proof of claim \eqref{E:volume} until the end. Assuming~\eqref{E:volume} for now, we have, 
\begin{align*}
\sig\big(\rho \in \SNH:\, G^{-t}(\rho) \in A,\text{ for some }t\in[T,T+\delta]\big)
&=\sig\Big(\bigcup_{t=T}^{T+\delta}G^t(A)\cap \SNH\Big)\\
&\leq C\int_A|\det  J_T(\rho)|d\sig(\rho).
\end{align*}

{Note that since $t\mapsto G^t$ is a smooth group, for $\delta>0$ small enough and $t\in[ T,T+\delta]$, 
\begin{equation}\label{E:2}
|\det  J_t(\rho)|\leq 2|\det J_T(\rho)|.
\end{equation}
}
Hence, 
\begin{align*}
&\sum_{n>0}\sig  \big(\rho \in A:\; G^{-t}(\rho) \in \SNH,\text{ for some }t\in[n\delta,(n+1)\delta]\big)\leq \\
&\qquad\qquad \leq C \sum_{n>0}\int_A|\det  J_{n\delta}(\rho)|\;d\sig(\rho)\\
&\qquad\qquad \leq 2C\delta^{-1}\int_0^\infty\int_A|\det  J_t(\rho) |\;d\sig(\rho) dt<\infty.
\end{align*}

Therefore, by the Borel--Cantelli Lemma, 
\[
\sig\big(\rho \in A:\;  G^{-t}(\rho) \in \SNH \text{ for infinitely many }t\in[0,\infty)\big)=0
\]
and in particular, $\sig(\mc{L}_H^{-\infty}\cap A)=0.$ The case of~\eqref{e:spider2} is identical.

In order to finish the proof of the lemma we need  to establish the claim in \eqref{E:volume}. We proceed to do this. 
Fix $\e>0$. Let $\{A_{i,\e}\}_{i=1}^{N(\e)}$ be a partition of $A\subset \SNH$ into sets of radius less than $\e$.
Then for each $i$, there exists $\rho_i\in A_{i,\e}$ so that for $\rho \in A_{i,\e}$, 
\begin{align*}
G^t(\rho)&=G^{t}(\rho_i)+dG^{t} (\rho-\rho_i)+O(\e^2) \\
&=G^{t}(\rho_i)+dG^{t} ({\pi_i}(\rho-\rho_i))+O(\e^2),
\end{align*}
where $\pi_i:T_{\rho_i}(S^*H) \to T_{\rho_i}(\SNH) $ is the projection operator and $\rho-\rho_i$ is regarded as a vector in $T_{\rho_i}(\SNH)$. 
{Therefore,} 
\begin{multline*}
\sig\Big(\bigcup_{t=T}^{T+\delta} G^t(A_{i,\e})\Big)\leq  \\
\sup_{t\in[T,T+\delta]}|\det J_t(\rho_i)|\cdot \sig(A_{i,\e})(1+O(\e)){\sup_{\rho\in A_{i,\e}}\#\{t\in [T,T+\delta]: G^{t}(\rho)\in \SNH\}}.
\end{multline*}
Now, Proposition \ref{p:aardvark} together with the compactness of $\SNH$ give that for $\delta>0$ small enough and all $\rho \in A$,  
\[
\#\{t\in [T,T+\delta]:\; G^t(\rho)\in \SNH\}\leq 1.
\]

In particular, 
\begin{align*}
\sig\Big(\bigcup_{t\in[T,T+\delta]}G^t(A)\Big)&\leq  \sum_{i}\sig\Big(\bigcup_{t\in [T,T+\delta]}G^t(A_{i,\e})\Big)\\
&\leq \sum_{i,j}\sup_{t\in [T,T+\delta]}|\det  J_{t}(\rho_i)| \cdot \sig(A_{i,\e})(1+O(\e))\\
&\leq \sum_{i,j}2|\det  J_{T}(\rho_i)| \cdot \sig(A_{i,\e})(1+O(\e))
\end{align*}
where in the last line we use~\eqref{E:2}.

Sending $\e \to 0$, since $dG^t$ is continuous, the Dominated Convergence Theorem shows that 
\begin{equation}\label{E:1}
\sig\Big(\bigcup_{t=T}^{T+\delta}G^t(A)\Big)\leq \int_A 2|\det  J_T(\rho)|\, d\sig.
\end{equation}
as desired.
\end{proof}

\subsection{Manifolds with {no focal points} or Anosov flow} \label{S:parts4}
This section is dedicated to the proof of Theorem \ref{T:tangentSpace}. 
{In order to prove Theorem \ref{T:tangentSpace} we need a preliminary lemma in which we show, loosely speaking, that if $\rho_0 \in \SNH$ is a loop direction for which $G^{t_0}(\rho_0)\in \SNH$, then it suffices to find a tangent direction  $\mathbf{w}\in T_{\rho_0}\SNH$ with the property that $dG^{t_0}(\mathbf{w})$ is not tangent  to $\SNH$ to ensure that $G^t(\rho) \notin \SNH$ for almost every $\rho \in \SNH$ so that $(t,\rho)$ is near $(t_0, \rho_0)$.}

\begin{lemma}
\label{l:implicit}
Suppose that $\rho_0\in \SNH$ with $G^{t_0}(\rho_0)\in \SNH$ for some $t_0>0$. If there exists $\mathbf{w}\in T_{\rho_0}\SNH$ with ${dG^{t_0}\mathbf{w}}\notin T_{G^{t_0}\rho_0} \SNH\oplus \R H_p$,  then there exists $U_{t_0,\rho_0} \subset \R\times\SNH$ a neighborhood of $(t_0,\rho_0)$ for which
\[
\sig\Big( \rho \in \SNH:\;{ \text{ there exists }t\text{ with }(t,\rho)\in U_{t_0,\rho_0} \text{\;and\; } G^t(\rho)\in \SNH}\Big)=0.
\]
\end{lemma}

\begin{proof}

We use the Implicit Function Theorem. Define
\[
\psi:\R\times \SNH \to \SM, \qquad \quad  \psi (t,\rho)=G^t(\rho),
\]
so that 
\[
d\psi_{(t,\rho)}(\tau,w)=\tau H_p{(G^t(\rho))}+dG^t_{\rho}w.
\]
and let $ {f_1,\dots f_{n}}\in C^\infty(\SM;\R)$ be defining functions for $\SNH$ near $G^{t_0}(\rho_0)$. In particular, 
\[
\SNH=\bigcap_{i=1}^{{n}}\{f_i=0\},\qquad \{df_i\}\text{ are linearly independent on }\SNH.
\]
Finally, let $F\in C^\infty(\SM;\R^{n-1})$ be given by
\[
F=(f_1,\dots, { f_{n}}).
\]
Note that  $G^t(\rho)\in \SNH$ if and only if $(t, \rho) \in (F\circ \psi)^{-1}(0)$. Now, since ${dG^{t_0}\mathbf{w}}\notin T_{G^{t_0}\rho_0} (\SNH)\oplus \R H_p$,  Proposition~\ref{p:aardvark} gives that
the vectors
\[
d(F\circ\psi)_{(t_0,\rho_0)}(0,{\bf{w}})=dF_{G^{t_0}(\rho_0)}\left(d\psi_{(t_0,\rho_0)}(0,{\bf{w}})\right)
\]
and
\[
d(F\circ\psi)_{(t_0,\rho_0)}(\tau,0)=dF_{G^{t_0}(\rho_0)}\left(d\psi_{(t_0,\rho_0)}(\tau,0)\right)
\]
are linearly independent.
We then have that 
\[
\text{rank}(d(F\circ\psi)_{(t_0,\rho_0)})\geq 2.
\]

By the implicit function theorem, there is a neighborhood $U$ of $(t_0,\rho_0)$, a neighborhood $V\subset \R^{\ell}$ of $0$ for some $\ell \leq n-2$, and smooth functions $s:\SNH\to \R$, $\alpha:V\to \SNH$ with $s(0)=t_0$, $\alpha(0)=\rho_0$, so that 
\[
U \cap \psi^{-1}( \SNH)=\{ (s(\alpha(q)),\alpha(q)):\; q\in V\}.
\]
In particular, since $\dim V<n-1=\dim(\SNH)$, 
\[
\sig\Big( \rho \in \SNH:\;  \text{ there exists }t\text{ such that }(t,\rho)\in U,\,G^t(\rho)\in \SNH \Big)=0,
\]
as claimed.
\end{proof}

{Next we present two propositions in which we show that if $(M,g)$ has no focal points or Anosov geodesic flow, then for any compact subset $K \subset \SNH \backslash \mc{N}_H$ there is {a decomposition of $K$, $K=K^+\cup K^-$ and} $T$ sufficiently large so that if $\rho_0 \in K^{{\pm}}$ and $G^{t_0}(\rho_0) \in \SNH$ with either ${\mp} t_0>T$, then   there exists $\mathbf{w}\in T_{\rho_0}\SNH$ with $ {dG^{t_0}\mathbf{w}}\notin T_{G^{t_0}\rho_0} \SNH\oplus \R H_p$.  This will allow us to later use Lemma \ref{l:implicit} to prove Theorem \ref{T:tangentSpace}.}
{We define the following functions $m,m_{\pm}:\SNH\to \{0,\dots,n-1\}$ 
\begin{equation}
\label{e:dim}
\begin{gathered}
m(\rho):=\dim (N_+(\rho)+N_-(\rho)),\qquad m_{\pm}(\rho):=\dim N_{\pm}(\rho)
\end{gathered}
\end{equation}
and note that the continuity of $E_{\pm}(\rho)$ implies that $m,\,m_{\pm}$ are upper semicontinuous.
}

\begin{proposition}\label{P:1} 
{Suppose $(M,g)$ has Anosov geodesic flow and l}et $K \subset \SNH \backslash {\mc{S}_H}$ be a compact set. There exist positive constants $T, \e>0$ so that if $\rho_0 \in K$, $|t_0| \geq T$, and 
\[
G^{t_0}(\rho_0) \in \; \overline{B(\rho_0, \e)} \cap \SNH,
\]
then 
{there is $\mathbf{w}\in T_{\rho_0}(\SNH)$ with
\begin{equation}
\label{e:noTangent}
dG^{t_0} (\mathbf{w})\notin T_{G^{t_0}(\rho_0)}(\SNH)\oplus \re H_p.
\end{equation}}
\end{proposition}

\begin{proof}
 Let $\rho_0\in K$. Since
 $T_{\rho_0}(\SNH)\neq N_+(\rho_0)\oplus N_-(\rho_0)$,
we may choose
 \[
 {\bf u} \in T_{\rho_0}(\SNH) \setminus ( N_+(\rho_0)\oplus N_-(\rho_0)),\quad \|{\bf u}\|=1.
 \]
Now, let $\mathbf{u}_+\in E_+(\rho_0)$ and $\mathbf{u}_-\in E_-(\rho_0)$ be so that 
\[
\mathbf{u}=\mathbf{u}_++\mathbf{u}_-.
\]
Without loss of generality, we assume that $\mathbf{u}_-$ is orthogonal to $N_-(\rho_0)$ and, since $\rho_0$ varies in a compact subset of $\SNH\backslash \mc{A}_H$, we may assume 
uniformly for $\rho_0\in K$ that
\[
M^{-1}\|\mathbf{u}_+\|\leq \|\mathbf{u}_-\|\leq M\|\mathbf{u}_+\|.
\]

Since $dG^t:E_-(\rho_0)\to E_-(G^t(\rho_0))$ is an isomorphism,
\[
\dim \Span\begin{pmatrix} dG^t(\mathbf{u}_-), &dG^t(N_-(\rho_0))\end{pmatrix}=1+\dim N_-(\rho_0).
\]
Note that for $m_-$ as in~\eqref{e:dim}, $m_-$ is upper semicontinuous and we may choose $\e>0$ uniform in $\rho_0 \in \SNH$, so that $\dim N_-(G^t(\rho_0))\leq \dim N_-(\rho_0)$ for all $t$ such that $G^t(\rho_0) \in B(\rho_0, \e).$
For such values of $t$ we then have

\begin{equation}\label{E:dimension}
\dim \Span\begin{pmatrix} dG^t(\mathbf{u}_-), &dG^t(N_-(\rho_0))\end{pmatrix}\geq1+\dim N_-(G^t(\rho_0)). \medskip
\end{equation}

Next, we note that $\Span\!\begin{pmatrix} dG^t(\mathbf{u}_-), &dG^t(N_-(\rho_0))\end{pmatrix} \subset E_-(G^t(\rho_0))$. Also, note that if {$dG^t({\bf w}) \in E_-(G^t(\rho_0)) \backslash N_-(G^t(\rho_0))$, then $dG^t({\bf w})  \notin T_{G^t(\rho_0)}(\SNH)$}. 
In particular, relation \eqref{E:dimension} gives that there exists a linear combination
\[
{\bf w_t}= a_t \,\mathbf{u}_- +  {\bf e}_-(t),
\]
with ${\bf e}_-(t) \in N_-(\rho_0)$, so that 
\[
\left \|  \pi_{t,\rho_0} (dG^t {\bf w_t}) \right\|=1=\left \|   dG^t {\bf w_t} \right\|,
\]
where $\pi_{t,\rho_0}: T_{G^t(\rho_0)}(S^*M) \to V_{t,\rho_0}$ is the orthogonal projection map onto  a subspace  $V_{t,\rho_0}$ of  $T_{G^t(\rho_0)}(S^*M)$ chosen so that $T_{G^t(\rho_0)}(S^*M)=V_{t,\rho_0} \oplus T_{G^t(\rho_0)}(\SNH)$ {is an orthogonal decomposition}. 
{If we had that ${\bf w_t}$ was a tangent vector in $T_{G^t(\rho_0)}(S^*M)$, then we would be done. However, since $\mathbf{u}_-$ is not necessarily in $T_{G^t(\rho_0)}(S^*M)$ we have to modify ${\bf w_t}$ a bit.}
Consider the vector
\[
{\bf \tilde{w}_t}= a_t\, \mathbf{u} + {\bf e}_-(t),
\]
and note that ${\bf \tilde{w}_t} \in T_{\rho_0}(\SNH)$. 
Then, 
\[dG^t ( {\bf \tilde{w}_t})=  dG^t({\bf w_t})+a_t\,  dG^t (\mathbf{u}_+).  \]

By the definition of Anosov geodesic flow, for all $\delta>0$, there exists $T=T(\delta)>0$ so that
\[
\| (dG^t|_{E_-})^{-1}\|\leq \delta,\quad t\geq T.
\]

Thus, since ${\bf w_t}\in E_-(\rho_0)$ and {$\|{\bf w_t}\|\leq \delta$}, we have
\[
|a_t|\leq \delta\|\mathbf{u}_-\|^{-1},\qquad { t\geq T}.
\]
{Observe next, \cite[Corollary 2.14]{Eberlein73} that there exists $B>0$ uniform in $TS^*M$ so that for $v\in E_+(\rho)$, {and $t\geq 0$} $\|dG^tv\|\leq B\|v\|.$}
In particular, {choosing $\delta<\frac{1}{2B}\|{\bf u}_-\|\|{\bf u}_+\|^{-1}$, for $t>T(\delta,K)$,}
\[
\|  \pi_{t,\rho_0} (dG^t  {\bf \tilde{w}_t}) \| \geq \|  \pi_{t,\rho_0} (dG^t {\bf w_t})\|-\| a_t\,  \pi_{t,\rho_0} (dG^t {\bf u}_+) \| >\frac{1}{2}.
\]

Hence, there exists $\e>0$ and $T>0$ (uniform for $\rho_0\in K$) so that  if $G^{t_0}(\rho_0)\in \SNH\cap B(\rho_0,\e)$ for some $t_0$ with $|t_0|>T$, {then there is $\mathbf{w}=\tilde{\mathbf{w}}_{t_0}\in T_{\rho_0}(\SNH)$} so that 
\begin{equation}
\label{e:noTangent}
dG^{t_0} (\mathbf{w})\notin T_{G^{t_0}(\rho_0)}(\SNH)\oplus \re H_p.
\end{equation}
\end{proof}

 We {next introduce} the following result in which we show that {for manifolds with Anosov geodesic flow} the set of points in ${\mc R}_H{\cap[\mc{S}_H\setminus \mc{M}_H]}$ has measure zero.
\begin{lemma}
\label{l:noStable}
Suppose that $(M,g)$ has Anosov geodesic flow. Then
\[
\sig\Big(\mc{R}_H\cap \{ \rho \in \SNH:\; T_\rho(\SNH) \subset E_+(\rho)\}\Big)=0
\]
and 
\[
\sig\Big(\mc{R}_H\cap \{ \rho \in \SNH:\; T_\rho(\SNH) \subset E_-(\rho)\}\Big)=0.
\]
{In particular, $\sig\big(\mc{R}_H\cap [\mc{S_H\setminus \mc{M}_H]}\big)=0.$} 
\end{lemma}

\begin{proof}
Observe that setting
\[
A:=\{ \rho \in \SNH:\;  T_\rho(\SNH) \subset E_+(\rho)\},
\]
we have
\[
|\det J_t(\rho)|\leq C^{n-1}e^{-(n-1)t/C},\qquad  t\geq 0, \medskip
\]
fot $J_t(\rho)$ defined in \eqref{E:J}.
It follows that 
\[
\vol(G^t(A)) \leq C^{n-1}e^{-(n-1)t/C} \sig(A),
\] and so
\[
\int_0^\infty\vol(G^t(A)) dt<\infty.
\]
Therefore, the proof is complete by Lemma~\ref{P:integral}. The $E_-$ case is identical where we integrate backwards in time rather than forwards.
\end{proof}

In what follows we write 
\[{\mc{M}^\pm_H}:=\Big\{\rho\in \SNH:\; \,N_{\pm}(\rho)\neq\{0\}\Big\},\]
and note that 
\[\mc{N}_H= \mc{S}_H \cup ({\mc{M}}_H^+ \cap {\mc{M}}_H^-).\]
Note that 
$$
\SNH\setminus \mc{N}_H =\big[\SNH\setminus ({\mc{S}}_H\cup {\mc{M}}_H^+)\big]\bigcup \big[\SNH\setminus ({\mc{S}}_H\cup {\mc{M}}_H^-)\big].
$$
\begin{proposition}\label{P:focal} 
{Suppose $(M,g)$ has no focal points and l}et $K \subset \SNH\setminus ({\mc{S}}_H\cup {\mc{M}}_H^{\pm})$ be a compact set. There exist positive constants $T, \e>0$ so that if $\rho_0 \in K$, ${\mp}t_0\geq T$, and 
\[
G^{t_0}(\rho_0) \in \; \overline{B(\rho_0, \e)} \cap \SNH,
\]
then 
{there is $\mathbf{w}\in T_{\rho_0}(\SNH)$ with
\begin{equation}
\label{e:noTangent}
dG^{t_0} (\mathbf{w})\notin T_{G^{t_0}(\rho_0)}(\SNH)\oplus \re H_p.
\end{equation}}
\end{proposition}

\begin{proof}
We prove the lemma for $K\subset\SNH\setminus ({\mc{S}}_H\cup {\mc{M}}^-_H)$, the other case follows similarly after sending $t\to -\infty$ rather than $t\to \infty$. 

Define {$\mathcal C_+^\e(\rho)\subset T_\rho S^*\!M$} as the conic set of vectors forming at least an $\e>0$  angle  with $E_+(\rho)$.
Since $m$ is upper semicontinuous, $E_+$ is continuous, {and $T_{\rho}\SNH\neq N_+(\rho)+N_-(\rho) $}, there exists $\ep>0$ so that   $ T_{\rho}\SNH \cap\mathcal C_+^\e(\rho) \neq 0$ for all $\rho \in K$.   
 
Next, let $\rho_0 \in K$. Since  $N_-(\rho_0)=\{0\}$, the upper semicontinuity of $m_-$ implies that  $N_-(\rho)=\{0\}$ for all $\rho\in B(\rho_0,\e)$, after possibly shrinking $\ep$.
In particular, the continuity of $E_-$  implies that there exists $\delta>0$ so that for $\rho \in B(\rho_0,\e)$, the angle between $E_-(\rho)$ and $T_\rho \SNH$ is larger than $\delta$ (after possibly shrinking $\ep$).

We claim that for $ w \in \mathcal C_+^\e(\rho_0) \backslash\{0\}$, there exists $T=T(\delta,\e)$ so that for $t\geq T$,
\begin{equation}
\label{e:toE-}
\text{dist}\Big(\frac{dG^tw}{\|dG^tw\|}\,,\, E_-(G^t(\rho_0))\Big)\leq \delta.
\end{equation}
The proof of \eqref{e:toE-} is postponed until the end.

{To finish the argument we argue by contradiction. Suppose that for  $t_0\geq T$, we have $G^t(\rho_0)\in B(\rho_0,\e)$ and
\[dG^{t_0}(T_{\rho_0}\SNH)=T_{G^{t_0}(\rho_0)}\SNH.\]  
Then, using that  $ T_{\rho_0}\SNH \cap\mathcal C_+^\e(\rho_0) \neq 0$, we conclude from the claim in \eqref{e:toE-} applied to some $ w \in T_{\rho_0}\SNH \cap\mathcal C_+^\e(\rho_0) \backslash\{0\}$  that there exists $v \in E_-(G^t(\rho_0))$ so that the angle between $v$ and $\frac{dG^tw}{\|dG^tw\|} \in T_{G^{t_0}(\rho_0)}\SNH$ is smaller than $\delta$. In particular, setting $\rho:=G^{t_0}(\rho_0) \in B(\rho_0,\e)$ we conclude that the angle between $T_\rho\SNH$ and $E_-(\rho)$ is smaller than $\delta$. And this is a contradiction since $ \rho \in B(\rho_0,\e)$. This conclude{s} the proof of the proposition {once we have~\eqref{e:toE-}}.

It only remains  to prove the claim in \eqref{e:toE-}.
 Let $w \in C_+^\e(\rho_0) \backslash\{0\}$. Then we can write \[w={\tilde{u}}_++{\tilde{v}}\] with ${\tilde{u}}_+\in E_+(\rho_0)$ and ${\tilde{v}}\in \tilde{V}(\rho_0)$, where  $\tilde{V}(\rho_0)\subset T_\rho S^*M$ denotes the collection of vertical vectors in $T_{\rho_0} \SNH$ orthogonal to $H_p$. 
Note that there exists $c_\e>0$ depending only on $\e$ so that 
\[  c_\e\|\tilde{u}_+\|\leq \|w\| \leq  \frac{1}{c_\e} \|\tilde v\|.\]

For any $e_t \in  E_-(G^t(\rho_0)) $ we decompose 
 \begin{equation}\label{e:triangle}
  \left\|\frac{dG^t{w}}{\|dG^tw\|} -e_t \right\|
   \leq
  \left\|\frac{dG^t\tilde{u}_+}{\|dG^tw\|} \right\|+  \left\|\frac{dG^t\tilde{v}}{\|dG^tw\|}-\frac{dG^t\tilde{v}}{\|dG^t\tilde{v}\|} \right\| +\left\|\frac{dG^t\tilde{v}}{\|dG^t\tilde{v}\|}-e_t \right\|, 
 \end{equation}
 and {find $e_t\in E_-(G^t(\rho_0))$ so that} each term in the RHS has size smaller than $\delta/3$.}
 
 {Note that since $\tilde{v}$ is vertical, the Jacobi field through $G^t(\rho)$ with initial conditions given by $J(0)=(dG^t\tilde{v})^h$ and $\dot{J}(0)=(dG^t\tilde{v})^v$, where $()^h$ and $()^v$ denote respectively the horizontal and vertical parts, has $J(-t)=0$ and hence, by~\cite[Remark 2.10]{Eberlein73}, there exists $T_1=T_1(\delta)>0$ so that for $ G^t\rho$ in a compact set, 
 $$\textup{dist}(dG^t\tilde{v}/\|dG^t\tilde{v}\|,E_-(G^t(\rho)))<\delta/3.$$ }
 
{In particular, for all $t \geq T_1$, {there exists} $e_t \in  E_-(G^t(\rho_0)) $ {so that } 
\begin{equation}\label{e:triangle1}
\left\|\frac{dG^t\tilde{v}}{\|dG^t\tilde{v}\|}-e_t \right\|\leq \frac{\delta}{3}.
\end{equation}}

Next, observe that by~\cite[Remark 2.10]{Eberlein73}, for all $\alpha>0$, there exists $T_2=T_2(\alpha)$ so that for all $\rho$, and $|t|\geq T_2$,
\begin{equation}
\label{e:uniformGrowth}
\|dG^{-t}|_{dG^t\tilde{V}(\rho)}\|\leq \alpha.
\end{equation}
In particular, by~\eqref{e:uniformGrowth}, given $R>0$ there exists $T_3=T_3(R,\e)>0$ so that for $|t|\geq T_3$ and $z \in C_+^\e(\rho_0) \backslash\{0\}$,
$$
 \| dG^t z\|\geq R\| z\|.
$$
Furthermore, by~\cite[Corollary 2.14]{Eberlein73}, there exists $B>0$ so that for all $t\geq 0$ and all $u\in E_{+}(\rho_0)$,
\begin{equation}
\label{e:forwardBound}
\|dG^tu\|\leq B\|u\|.
\end{equation}
{In particular, setting $R_{\delta,\e}:=3Bc_\e^{-1}\delta^{-1}$, and letting $|t| \geq T_3(R_{\delta,\e},\e)$,
\begin{equation}\label{e:triangle2}
 \left\|\frac{dG^t\tilde{u}_+}{\|dG^tw\|} \right\| \leq  \frac{ B\|\tilde u_+\|}{\|dG^tw\|} \leq  \frac{B\|\tilde u_+\|}{R_{\delta,\e} \|w\| }\leq \frac{\delta}{3}.
\end{equation}
On the other hand, for $|t| \geq T_3(R_{\delta,\e},\e)$,
\begin{equation}\label{e:triangle3}
\Big\|\frac{dG^t\tilde{v}}{\|dG^t\tilde{v}\|}-\frac{dG^t\tilde{v}}{\|dG^tw\|}\Big\|
=\frac{1 }{\|dG^t{w}\|}| \|dG^t\tilde{v}\|-\|dG^tw\| | 
\leq   \frac{\|dG^t{\tilde u_+}\| }{\|dG^t{w}\|} \leq \frac{\delta}{3}.
\end{equation}
Taking $T=\max\Big(T_3(R_{\delta,\e},\e), T_1(\delta)\Big)$ we conclude that the claim in \eqref{e:toE-} holds after combining \eqref{e:triangle1},\eqref{e:triangle2}, and \eqref{e:triangle3}, into \eqref{e:triangle}.}

\end{proof}


Now that we have introduced Propositions \ref{p:aardvark}, \ref{P:focal}, and \ref{P:1}, we are ready to present the proof of Theorem  \ref{T:tangentSpace}.\\


\noindent{\bf Proof of Theorem  \ref{T:tangentSpace}.}
We start with the case in which $(M,g)$ has no focal points. Recall, that $m,m_{\pm}$ from~\eqref{e:dim} are upper semicontinuous. In particular, the sets 
\[
\SNH\setminus {\mc{S}}_H=\{\rho \in \SNH: m(\rho)<n-1\} \quad  and  \quad \SNH\setminus {\mc{M}}_{H}^{\pm}=\{\rho \in \SNH: m_{\pm}(\rho)<1\}
\]
 are open, and hence $\SNH\setminus ({\mc{S}}_H\cup {\mc{M}}_H^\pm)$ are open as well. Thus, there exist collections $\{K^{\pm}_\ell\}_\ell$ of compact sets 
\[
K^+_\ell \subset \;  \SNH\setminus ({\mc{S}}_H\cup {\mc{M}}_H^+),\qquad K^+_\ell \subset \;  \SNH\setminus ({\mc{S}}_H\cup {\mc{M}}_H^-)
\]
 with 
\[
\sig(K^{\pm}_\ell )\;\uparrow\; \sig(\SNH\setminus {\mc{S}}_H\cup {\mc{M}}_H^{\pm}).
\]
Since 
$$
\SNH\setminus ({{\mc{S}}_H \cup ({\mc{M}}_H^+\cap {\mc{M}}_H^-)})=\big[\SNH\setminus ({\mc{S}}_H\cup {\mc{M}}_H^+)\big]\bigcup \big[\SNH\setminus ({\mc{S}}_H\cup {\mc{M}}_H^-)\big],
$$
the proof of the lemma will follow once we prove that for any compact subset $K^{\pm}\subset \SNH\setminus ({\mc{S}}_H\cup {\mc{M}}_H^{\pm})$ 

\begin{equation}\label{E:compact}
\sigma_H( \mc{R}_H\cap K^{\pm})=0.
\end{equation}
We then proceed to prove \eqref{E:compact}.

Let $T_{\pm}>0$ and  $\e>0$ be the constants associated to $K^{\pm}$ given by Proposition \ref{P:focal}. Since  
\[ 
\mc{R}_H \subset \Big[\bigcap_{m>0} \bigcup_{n \geq m} A_{n}^{\ep}\Big]\bigcap\Big[\bigcap_{m>0} \bigcup_{n \geq m} A_{-n}^{\ep}\Big],
\]
with 
\[
A_{n}^{\ep}:=\left \{ \rho \in \SNH :\;\; G^t(\rho) \in \overline {B(\rho, \ep)} \;\;\;\text{for some}\;\;t \in [n , n+1]\right\},
\]
we have that \eqref{E:compact} is a consequence of showing that 

\begin{equation}\label{E:compact2}
\sig( A_{n}^{\ep} \cap K^{\pm})=0,
\end{equation}
{for all $n$ with $\mp n\geq T_\pm$}.

To prove \eqref{E:compact2} let $\rho_0 \in A_{n}^{\ep} \cap K$. Since $G^{t_0}(\rho_0) \in \overline {B(\rho_0, \ep)}$ for some $t_0 \in [n , n+1]$, and $\mp t_0 \geq T$, Proposition \ref{P:focal} combined with Lemma \ref{l:implicit} give that there exists  $U_{t_0,\rho_0} \subset \R\times \SNH$ a neighborhood of $(t_0,\rho_0)$ for which
\[
\sig \left( \rho \in \SNH: \; G^t(\rho) \in \SNH \;\;\;\text{for some}\;\; (t,\rho)\in U_{t_0,\rho_0}\Big. \right)=0.
\]
Since, $K^{\pm}$ is compact if $A_n^{\e}$ is closed, $A_{n}^{\ep} \cap K^{\pm}$ is compact and we can cover $[n,n+1]\times (K^{\pm}\cap A_n^\e)$ by finitely many such neighborhoods and in particular, 
\[
\sig \left( \rho \in \SNH: \; G^t(\rho) \in \SNH \;\;\;\text{for some}\;\; (t,\rho)\in [n,n+1]\times (K^{\pm}\cap A_n^\e)\Big. \right)=0.
\]
and hence $\sig (A_n^\e\cap K^{\pm})=0$. Therefore, we have \eqref{E:compact2} provided we show that $A_{n}^{\ep}$ is closed 

We dedicate the end of the proof to showing that $A_{n}^{\ep}$ is closed. To see this, let $\{\rho_j\} \subset A_{n}^{\ep}$ {with $\rho_j\to\rho \in \SNH$.}
For each $j$ let $t_j \in [n, n+1]$ be so that $G^{t_j}(\rho_j) \in \overline{B(\rho_j, \ep)}$. By possibly taking a subsequence of times, we may assume that there exists $t \in  [n, n+1]$ with the property that $t_j \to t$ as $j \to \infty$.
 In particular, we have that $G^{t_j}(\rho_j) \to G^{t}(\rho)$. Then, the triangle inequality
\[
d(G^{t}(\rho), \rho) \leq \limsup_{j \to \infty} \left( d(\rho, \rho_j ) + d(\rho_j, G^{t_j}(\rho_j))+d(G^{t_j}(\rho_j), G^t(\rho)) \right) \leq \e
\]
shows that $\rho \in A_{n}^{\ep, }$ as claimed.

In the case that $(M,g)$ has Anosov geodesic flow, we simply appeal to Proposition~\ref{P:1} in place of Proposition~\ref{P:focal} to show that, for $K\subset \SNH\setminus {\mc{S}}_H$ compact, 
$$
\sig(K\cap \mc{R}_H)=0.
$$
and hence using that $\SNH\setminus {\mc{S}}_H$ is open and approximating $\SNH\setminus {\mc{S}}_H$ by compact sets, {we see that $\sig(\mc{R}_H\setminus \mc{S}_H)=0$.} {Then, applying Lemma~\ref{l:noStable}, $\sig\big(\mc{R}_H\cap [\mc{S}_H\setminus \mc{M}_H]\big)=0$ and} the theorem follows.
\qed

\subsection{Proof of parts {\bf \ref{a3}},  {\bf \ref{a4}}, {\bf \ref{a5}} and {\bf \ref{a6}}}  \label{S:parts3}
 Since in all these cases $(M,g)$ has Anosov flow,  for all $\rho \in \SM$,
\[
T_\rho(\SM)=E_+(\rho)\oplus E_-(\rho)\oplus \re H_p.
\]
where $E_-,E_+$  are stable and unstable directions as before. Moreover, there exists $C>0$ so that for all $\rho\in \SM$,
\begin{align*}
 &|dG^t(v)|\leq Ce^{- t/C}|v| \;\;\qquad  \text{for}\;  v\in E_+\;\; \text{and}\; t\to  +\infty,\\
  &|dG^t(v)|\leq C\;e^{t/C}|v| \;\;\;\qquad  \text{for}\;  v\in E_-\;\; \text{and}\; t\to  -\infty.
\end{align*}

\begin{proof}[{\bf Proof of part \ref{a5}}] For this part we assume that $(M,g)$ has Anosov geodesic flow, {non-positive curvature}, and $H$ is totally geodesic.

We use that, since there are no parallel Jacobi fields on a manifold with {non-positive curvature} and  Anosov geodesic {flow~\cite[Theorem 1 (6)]{Eberlein73b}},  the spaces $E_+$ and $E_-$ are nowhere horizontal. In particular, for any horizontal vector $v^h$, $\|dG^tv^h\|\to \infty$ for $t\to \pm \infty$. 
To take advantage of this, fix $\rho=(x,\xi)\in \SNH$. Since $H$ is totally geodesic,  the horizontal lift  $v^h$ of any $v\in T_xH$  satisfies 
\[v^h \in T_{\rho}(\SNH).\]
On the other hand, $v^h\notin E_+(\rho)\cup E_-(\rho).$

Suppose that  $H$ is $n-1$ dimensional. Then, we may choose linearly independent vectors $\{v_1,v_2, \dots,v_{n-1}\} \in T_xH$ and get 
\[
T_{\rho}(\SNH)= \text{span}\{v_1^h,v_2^h, \dots, v_{n-1}^h\}.
\]
In particular, this yields that
\[
T_{\rho}(\SNH) \cap (E_+(\rho)\cup E_-(\rho))=\emptyset.
\]
 Therefore, 
 \[
{ \mc{S}_H}=\emptyset,
 \]
  and hence $\sig(\mc{R}_H)=0$. 

To finish the proof we explain that it suffices to assume that $H$ is $n-1$ dimensional. Note that since $H$ is totally geodesic submanifold,   $H_t:=\pi (G^t(\SNH))$ is also a totally geodesic submanifold. Now, for $t$ small, \[G^t:N^*\!H\to M\]
 is an isometry, and in particular, $H_t$ is an embedded submanifold of dimension $n-1$. Moreover, by Lemma~\ref{l:flowRecur}, $\sigma_{\SNH_t}(\mc{R}_{H_t})=0$ implies $\sig(\mc{R}_H)=0.$ 
Therefore, it is enough to show that $\sig(\mc{R}_H)=0$ for every   totally geodesic submanifold $H$ of dimension $n-1$ which we have already done.

\end{proof}


The proofs of Parts \ref{a3}, \ref{a4}, and \ref{a6}, rely on showing that in each of these settings one has that the set of  points $\rho \in {\mc{R}_H}$ for which  $T_\rho(\SNH) $ is purely stable, or purely unstable, has full measure and applying Lemma~\ref{l:noStable}}.
\begin{proof}[{\bf Proof of part \ref{a4}}] For this part we assume that $(M,g)$ is a surface with Anosov geodesic flow. Theorem~\ref{T:tangentSpace} implies 
\begin{equation*}
\sig(\mc{R}_H)=\sig\big(\mc{R}_H\cap {\mc{S}_H\cap \mc{M}_H}\big).
\end{equation*}
But, {since $\dim M=2$, we have $\dim \SNH=1$ and, since $E_+(\rho)\cap E_-(\rho)=\{0\}$, $\mc{M}_H=\emptyset$. Thus,} $\sig(\mc{R}_H)=0$ as claimed. 
\end{proof}


\begin{proof}[{\bf Proof of part \ref{a6}}]  For this part we assume that $(M,g)$ has {Anosov geodesic flow} and $H$ is a subset of a stable or unstable horosphere.
That $\sig({\mc R}_H)=0$ follows immediately from Lemma~\ref{l:noStable}.
\end{proof}

\ \\
\noindent{\bf Proof of part \ref{a3}.} 
{We start by showing that it suffices to assume that $H$ is $n-1$ dimensional. 
{Since the exponential map is  a radial isometry,
 $H_t=\{\exp_x(t\xi):\; (x, \xi) \in \SNH\}$
 is an embedded submanifold of dimension $n-1$ for small $t$.} Moreover, by Lemma~\ref{l:flowRecur}, $\sigma_{\SNH_t}(\mc{R}_{H_t})=0$ implies $\sig(\mc{R}_H)=0.$ 
Therefore, it is enough to show that $\sig(\mc{R}_H)=0$ for every submanifold $H$ of dimension $n-1$. }

We note that by Theorem \ref{T:tangentSpace} we have
\[
\sig(\mc{R}_H)=\sig\left({\mc{R}_H\cap \mc{S}_H\cap \mc{M}_H}\right)
\]

\begin{lemma}\label{P:hyperbolic}
Let {$(M,g)$ be a compact manifold with constant negative curvature and } $H \subset M$ be a {closed embedded} hypersurface. Then
\[
{\sig (\mc{S}_H\cap \mc{M}_H)=0.}
\]
\end{lemma}

Note that this result combined with Theorem~\ref{T:tangentSpace} yield that $\sig({\mc R}_H)=0$ finishing the proof of Part {\ref{a6}}.\\
\qed\ \\

{The rest of this section is dedicated to the proof of Lemma \ref{P:hyperbolic}.}
{Since we may work locally to prove Lemma~\ref{P:hyperbolic}, we lift the hypersurface $H$ to the universal cover  $\mathbb{H}^n$. Hence, in this section we work with the hyperbolic space
\[
\mathbb{H}^n=\Big\{(x_0, x_1, \dots, x_n)\in \R^{n+1}:\;  x_0>0, \;x_0^2-\sum_{i=1}^n x_i^2=1 \Big\}.
\]}
We endow $\mathbb{H}^n$ with the metric  $g=dx_0^2 - \sum_{i=1}^{n} dx_i^2$. 
To prove Lemma \ref{P:hyperbolic} we adopt the notation
\[
\langle v,w\rangle_g=v_0w_0-\sum_{i=1}^{n} v_iw_i
\]
 for the inner product induced by the metric $g$. We also write  $\langle v,w\rangle=v_0w_0+\sum_{i=1}^{n} v_iw_i$ for the usual inner product in $\R^{n+1}$.
With this notation the sphere bundle takes the form $S\mathbb{H}^n=\{(x,w):\;  x\in \mathbb{H}^n,w\in \mathbb{R}^{n+1},\langle w,w\rangle_g=-1,\langle x,w\rangle_g=0\}$, and its tangent space at $p=(x,w)$ can be decomposed into a direct sum $T_p(S\mathbb{H}^n)=E_+(p)\oplus E_-(p)\oplus \R X$ where the stable and stable fibers are $\tilde{E}_-(p)=\{(v,-v): \langle x,v\rangle_g=\langle w,v\rangle_g=0\}$ and  $\tilde{E}_+(p)=\{(v,v): \langle x,v\rangle_g=\langle w,v\rangle_g =0\}$ and $X$ is the generator of the geodesic flow. 
Since we work in the co-sphere bundle, we record the structure of the dual spaces.
The co-sphere bundle is
\[
S^*\mathbb{H}^n=\{(x,\xi):\;  x\in \mathbb{H}^n,\xi\in \mathbb{R}^{n+1},\langle \xi,\xi\rangle_g=-1,\langle x,\xi\rangle=0\},
\]
and the tangent space at any $\rho=(x,\xi)\in  S^*\mathbb{H}^n$ is
\[
T_\rho( S^*\mathbb{H}^n)=\{(v_x,v_\xi):\; \langle x,v_x\rangle_g=\langle \xi,v_x\rangle+\langle x,v_\xi\rangle=\langle \xi,v_\xi\rangle_g=0\}.
\]
We then have 
\[
T_\rho(S^*\mathbb{H}^n)=E_+(\rho)\oplus E_-(\rho)\oplus \R {H_p},
\] 
where
\begin{equation}\label{E:plus}
E_+(\rho)=\{((v_0,v'),(v_0,-v')):\;  \langle x,v\rangle_g=\langle\xi,v\rangle =0\}.
\end{equation}
and
\begin{equation}\label{E:minus}
E_-(\rho)=\{((v_0,v'),(-v_0,v')):\; \langle x,v\rangle_g=\langle \xi,v\rangle=0\}.
\end{equation}
Here, and in what follows, we adopt the notation $(z_0, z',z_d)$ to represent a point in $\R \times \R^{n-1}\times \R$.

\begin{proof}[{\bf Proof of Lemma \ref{P:hyperbolic}.}]

We assume that  $\gamma$ is a parametrization of $H \subset \mathbb{H}^n$ in a neighborhood $V \subset H$ of $y$. That is,     
\[
H \cap V=\{(\alpha(x'), x', \gamma(x')):  x' \in \tilde V\},
\]
for some $\tilde V \subset \R^{n-1}$ open, and where 
\[
\alpha(x'):=\sqrt{1+|x'|^2+\gamma(x')^2}.
\]
Using that $x_0-\alpha(x')$ and $x_n-\gamma(x')$ are defining functions for $H$ as a subset of $\R^{n+1}$ we find that 
\[
N^*\! H=\{ (\alpha,\; x',\; \gamma,\; - \lambda f\alpha,\;  \lambda (fx'-\partial\gamma),\; \lambda(f\gamma+1)):\; \lambda \in \R\},
\]
where to shorten notation we write 
\[
f:=\gamma-\langle x',\partial \gamma\rangle.
\]
This yields that 
\[
\SNH=\{ (\alpha,\; x',\; \gamma,\; - \lambda f\alpha,\;  \lambda (fx'-\partial\gamma),\; \lambda(f\gamma+1))\},
\]
where 
\[ 
\lambda:=(1+|\partial\gamma|^2+f^2)^{-\frac{1}{2}}.
\]
Therefore, given $\rho=(x,\xi)\in \SNH$ we find
\begin{equation}\label{E:tangent}
T_{\rho}(\SNH)=\{( \langle \partial\alpha, w\rangle,\; w,\;   \langle \partial\gamma, w\rangle,\;  \langle A, w\rangle,\;   \langle  B, w\rangle,\;   \langle  C, w\rangle):\; w\in \R^{n-1}\},
\end{equation}
where
\[
A:= -  \partial(\lambda f\alpha), \qquad  B:= \partial(\lambda (fx'-\partial\gamma)), \qquad   C:=\partial(\lambda(f\gamma+1)).
\]

We assume without loss of generality that $y=(\alpha(0),0, \gamma(0))$,  where $\gamma(0)=0$ and $\partial \gamma(0)=0$. Note that, with 
\[
\gamma(x')=\frac{1}{2}\langle Qx',x'\rangle +O(|x'|^3),
\] 
where $Q$ is an $(n-1)\times (n-1)$ symmetric matrix
we have
\begin{align*}
\alpha &=1+\frac{1}{2}|x'|^2+O(|x'|^4),&\partial \alpha& =x'+O(|x'|^3),\\
f &=-\frac{1}{2}\langle Qx',x'\rangle+O(|x'|^3),&\partial f& =-Q x'+O(|x'|^2),\\
\lambda&=1-\frac{1}{2}|Qx'|^2+O(|x'|^4),&\partial\lambda&=-\langle Q^2x',w\rangle+O(|x'|^3).
\end{align*}

Now, suppose there exist  two non-zero vectors
\[
X_+ \in E_+(\rho)\cap T_{\rho}(\SNH) \qquad \text{ and}  \qquad X_-\in E_-(\rho)\cap T_{\rho}(\SNH).
\]  
Then, according to  \eqref{E:tangent}, \eqref{E:plus} and \eqref{E:minus} we have that there exist $w_+, w_- \in \R^{n-1}$ so that 
\[
X_{\pm}=( \langle \partial\alpha, w_\pm \rangle,\; w_\pm,\;   \langle \partial\gamma, w_\pm\rangle,\;  \langle A, w_\pm\rangle,\;   \langle  B, w_\pm \rangle,\;   \langle  C, w_\pm\rangle) 
\]
and satisfying
\begin{enumerate}
\item[i)]  $\langle \partial\alpha, w_\pm \rangle=  \pm \langle A, w_\pm\rangle$ \medskip
\item [ii)]  $w_\pm=  \mp \langle B, w_\pm\rangle$ \medskip
\item [iii)]   $\langle \partial\gamma, w_\pm\rangle=  \mp \langle C, w_\pm\rangle$ \medskip
\item [iv)] $ \langle x, X_{\pm} \rangle_g =0$\medskip
\item [v)] $ \langle \xi, X_{\pm} \rangle =0$.
\end{enumerate}

We proceed to showing that there cannot exist $w_\pm$ satisfying conditions $(i), (ii)$ and $(iii)$ for all $\rho=(x,\xi)$ in a subset of $S\!N^*\!H$ with positive measure on which $T_{\rho}(S\!N^*\!H)=N_{+}(\rho) \oplus N_{-}(\rho)$,  $ N_{+}(\rho) \neq \{0\}, $ and $  N_{-}(\rho) \neq \{0\}  $.
Indeed, conditions $(i), (ii)$ and $(iii)$ read
\begin{enumerate}
\item[i)] $\langle x',w_{\pm}\rangle =\pm \langle Qx',w_{\pm}\rangle+ O(|x'|^2)$ \medskip
\item [ii)] $w_{\pm}=\pm Qw_{\pm} \pm (\partial^3\gamma(0) x')w_{\pm} + O(|x'|^2)$ \medskip
\item [iii)] $\langle Qx',w_{\pm}\rangle =\pm \langle Q^2x',w_{\pm}\rangle + O(|x'|^2).$  \medskip
\end{enumerate}

These equations imply that $w_{\pm}=\pm Qw_{\pm}$ and so $Q^2w_{\pm}=w_{\pm}$.  Furthermore, we claim that we may assume that  $\partial^3\gamma(0)=0$. Indeed,    let $\rho \in \SNH$ be so that 
$T_{\rho}(S\!N^*\!H)=N_{+}(\rho) \oplus N_{-}(\rho)$. Then, if $w \in T_{\rho}(S\!N^*\!H)$, we may decompose $w$ it as $w=w_+ +w_-$ and use  that condition $(ii)$ gives $(\partial^3\gamma(0) x')w=0$. If we had that condition $(ii)$ holds on a set of $\rho$'s with positive measure, we must have that $\partial^3\gamma(0)=0$ since we just showed that condition $(ii)$ should also hold for all $w \in T_{\rho}(S\!N^*\!H)$.
 We then work with 
\[
\gamma(x')=\frac{1}{2}\langle Qx',x'\rangle +O(|x'|^4).
\]
From this we get the improved estimates
\[
f =-\frac{1}{2}\langle Qx',x'\rangle+O(|x'|^4) \qquad \text{and} \qquad \partial f =-Q x'+O(|x'|^3).
\]

We derive the contradiction from studying the second order terms in $w_\pm=  \mp \langle B, w_\pm\rangle$. Indeed, 
\[
\langle B, w_\pm\rangle=\pm Qw_{\pm} + D(w_\pm)+ O(|x'|^3),
\]
where
\[
D(w_\pm):=- \partial^4\gamma(0)x'^2w_{\pm} + \langle x', w_\pm \rangle (Qx' \mp x') -\frac{1}{2} \langle Qx', x'\mp Qx' \rangle w_{\pm},
\]
and  where $\partial^4\gamma(0)x'^2w_{\pm}$ denotes the vector whose $i$-th entry is given by $(\partial^4\gamma(0)x'^2w)_k=\frac{1}{12}\partial_{ijkl}\gamma(0)x_kx_lw_j$.
Since $D(w_\pm)$ is a second order term in $x'$, equation $w_\pm=  \mp \langle B, w_\pm\rangle$ gives that 
\[
D(w_\pm)=0.
\]
To take advantage of this condition, we assume without loss of generality that 
\[
Q=\left(\begin{array}{ccc}1 & 0 & 0 \\0 & -1 & 0 \\0 & 0 &  \tilde Q\end{array}\right),
\]
where $\tilde Q$ is an $(n-3)\times (n-3)$ matrix, and that
\[
w_+=(1, 0, \dots, 0) \qquad \text{and} \qquad w_-=(0,1, 0 \dots, 0).
\]

We now use that all the coordinates of the vectors $D(w_\pm)$ equal $0$. 
Making the second coordinate of the vector $D(w_+)$ equal to $0$ gives
\[
-   \frac{1}{12}\sum_{k, l=1}^n\partial_{2 1 kl}\gamma(0)x_kx_l -2x_1x_2=0, 
\]
while setting the first coordinate of the vector $D(w_-)$ equal to $0$ yields
\[
-   \frac{1}{12}\sum_{k, l=1}^n\partial_{1 2 kl}\gamma(0)x_kx_l +2x_1x_2=0.
\]
This concludes the proof since we cannot have the two relations holding simultaneously for $x'$ in a subset of $H$ that has positive measure. \\
\end{proof}

\bibliography{biblio.bib}
\bibliographystyle{amsalpha}
\end{document}